\documentclass[11pt]{article}

\usepackage[T1]{fontenc}
\usepackage[utf8]{inputenc}
\usepackage[margin=2.54cm]{geometry}
\usepackage{microtype}
\usepackage{comment}
\usepackage{amsmath,amssymb,amsthm,mathtools,mathrsfs}
\usepackage{thmtools}
\usepackage{thm-restate}
\numberwithin{equation}{section}
\allowdisplaybreaks

\usepackage{booktabs}
\usepackage{array,multirow}
\usepackage{float}
\usepackage{subcaption}
\usepackage[export]{adjustbox}
\usepackage{enumitem}
\usepackage{placeins}
\usepackage{sectsty}
\usepackage{graphicx}
\usepackage{esvect}
\usepackage{xcolor}
\usepackage[backref=section]{hyperref}

\definecolor{mycitegreen}{HTML}{00A99D}
\hypersetup{
   colorlinks=true,
  linkcolor={blue!85!black},
  citecolor={red!85!black},
  urlcolor=mycitegreen
}
\usepackage[noabbrev]{cleveref}

\usepackage{tikz}
\usetikzlibrary{   arrows.meta,
    calc,
    shapes.geometric,
    decorations.markings}
\tikzset{
    cluster/.style={
        draw,
        circle,
        line width=1.1pt,
       minimum size=10mm,
        inner sep=0pt,
        outer sep=0pt
    },
    aset/.style={
        draw,
        ellipse,
        line width=1.1pt,
        minimum width=2.25cm,
        minimum height=1.35cm,
        inner sep=0pt,
        outer sep=0pt
    },
    bridgepoint/.style={
        circle,
        fill=black,
        draw=black,
        minimum size=3.2pt,
        inner sep=0pt,
        outer sep=0pt
    },
    contactdot/.style={
        circle,
        fill=black,
        draw=black,
        minimum size=2.6pt,
        inner sep=0pt,
        outer sep=0pt
    },
    blackline/.style={
        draw=black,
        line width=2.1pt,
        shorten >=1.2pt,
        shorten <=1.2pt
    },
    blackarrowline/.style={
        draw=black,
        line width=2.1pt,
        shorten >=1.2pt,
        shorten <=1.2pt,
        postaction={
            decorate,
            decoration={
                markings,
                mark=at position 0.9 with {
                    \arrow{
                        Stealth[
                            length=2.1mm,
                            width=1.95mm
                        ]
                    }
                }
            }
        }
    },
    redline/.style={
        draw=red,
        line width=1.15pt,
        shorten >=0pt,
        shorten <=0pt,
        postaction={
            decorate
        },
        decoration={
            markings,
            mark=at position 0.55 with {
                \arrow{stealth}
            }
        }
    }
}

\tikzset{
    setcircle/.style={
        draw=black,
        circle,
        line width=1.35pt,
        minimum size=19mm,
        inner sep=0pt,
        outer sep=0pt
    },
  setcircleA/.style={
    draw=black,
    circle,
    line width=1.35pt,
    minimum size=15mm,
    inner sep=0pt,
    outer sep=0pt
},
    mainarrow/.style={
        draw=black,
        line width=2.4pt,
        -{Stealth[
            length=2.5mm,
            width=2.1mm
        ]},
        shorten >=3pt,
        shorten <=3pt
    },
   densecurve/.style={
    draw=black,
    line width=2.4pt
}
}

\newcommand{\manualstealtharc}[6]{%

    \draw[densecurve]
        ($(#1.center)+(#3:#2cm)$)
        arc[
            start angle=#3,
            end angle=#4,
            radius=#2cm
        ];

    \coordinate (arrowtip-#1) at
        ($(#1.center)+(#5:#2cm)$);

    \pgfmathsetmacro{\tangentangle}{#5+90+#6}

    \draw[
        draw=black,
        line width=1.1pt,
        -{Stealth[
            length=2.7mm,
            width=2.1mm
        ]}
    ]
        ($(arrowtip-#1)+(\tangentangle+180:0.16cm)$)
        --
        (arrowtip-#1);
}

\usepackage{listings}
\definecolor{codebg}{RGB}{248,248,248}
\definecolor{codeframe}{RGB}{220,220,220}
\definecolor{codecomment}{RGB}{0,96,96}
\definecolor{codestring}{RGB}{128,0,0}
\definecolor{codekeyword}{RGB}{40,40,180}

\lstdefinestyle{pythonstyle}{
    language=Python,
    basicstyle=\ttfamily\small,
    numbers=left,
    numberstyle=\scriptsize\color{gray},
    stepnumber=1,
    numbersep=8pt,
    breaklines=true,
    showstringspaces=false,
    columns=fullflexible,
    frame=single,
    rulecolor=\color{codeframe},
    backgroundcolor=\color{codebg},
    keywordstyle=\color{codekeyword},
    commentstyle=\itshape\color{codecomment},
    stringstyle=\color{codestring}
}
\sectionfont{\centering}

\theoremstyle{plain}
\newtheorem{theorem}{Theorem}[section]

\newtheorem{lemma}[theorem]{Lemma}
\newtheorem{proposition}[theorem]{Proposition}

\newtheorem{conjecture}[theorem]{Conjecture}

\newtheorem{claim}[theorem]{Claim}
\newtheorem{case}{Case}[section]

\theoremstyle{definition}
\newtheorem{definition}[theorem]{Definition}

\newtheorem{question}[theorem]{Question}

\theoremstyle{remark}

\crefname{theorem}{Theorem}{Theorems}
\crefname{lemma}{Lemma}{Lemmas}
\crefname{proposition}{Proposition}{Propositions}
\crefname{corollary}{Corollary}{Corollaries}
\crefname{conjecture}{Conjecture}{Conjectures}
\crefname{problem}{Problem}{Problems}
\crefname{claim}{Claim}{Claims}
\crefname{observation}{Observation}{Observations}
\crefname{fact}{Fact}{Facts}
\crefname{definition}{Definition}{Definitions}
\crefname{construction}{Construction}{Constructions}
\crefname{question}{Question}{Questions}
\crefname{setup}{Setup}{Setups}
\crefname{example}{Example}{Examples}
\crefname{remark}{Remark}{Remarks}
\crefname{section}{Section}{Sections}
\Crefname{section}{Section}{Sections}
\crefname{subsection}{Subsection}{Subsections}
\Crefname{subsection}{Subsection}{Subsections}
\crefname{subsubsection}{Subsubsection}{Subsubsections}
\Crefname{subsubsection}{Subsubsection}{Subsubsections}
\crefrangelabelformat{lemma}{#3#1#4--#5#2#6}

\makeatletter
\newcommand{\labelinthm}[1]{%
  \label{temp#1}%
  \protected@write\@auxout{}{%
    \string\newlabel{#1}{{\emph{\ref{temp#1}}}{\thepage}{\emph{\ref{temp#1}}}{temp#1}{}}%
  }%
}
\makeatother

\title{The exact total degree threshold for the square
of a Hamilton cycle in digraphs}
\author{
Zhilan Wang\thanks{School of Mathematics, Shandong University, Jinan 250100, China. Supported by the National Natural Science Foundation of China (No.~12571373).}
\and
Shuo Wei\footnotemark[1]
\and
Jin Yan\footnotemark[1]
\thanks{Corresponding author. Email: \href{mailto:yanj@sdu.edu.cn}{\texttt{yanj@sdu.edu.cn}}}
}
\date{}

\begin{document}
\maketitle
\vspace{-1em}
\begin{abstract}
\noindent The P\'{o}sa-Seymour conjecture establishes the minimum degree threshold required to guarantee the presence of the $k$th power of a Hamilton cycle in a graph. Following numerous partial results, Koml\'{o}s, S\'{a}rk\"{o}zy, and Szemer\'{e}di confirmed the conjecture holds for all sufficiently large graphs. Treglown later conjectured the analogous minimum semi-degree threshold for forcing the $k$th power of a Hamilton cycle in a digraph. Subsequently, DeBiasio et al. proposed a conjecture on the minimum total degree threshold for the same problem. In this paper we settle the conjecture of DeBiasio et al. for $k=2$. Specifically, we prove that every sufficiently large $n$-vertex
digraph with minimum total degree at least $8n/5-c$ contains the square of a Hamilton cycle, where $c=2$ if
$n\equiv2,4\pmod 5$, and $c=1$ otherwise.


\end{abstract}
\section{Introduction}
Hamiltonicity is one of the most central notions in graph theory, and it has been extensively studied by numerous researchers. Although deciding whether a graph is Hamiltonian is NP-complete, many important results provide sufficient conditions for Hamiltonicity. A classical theorem of Dirac \cite{Dirac} states that every graph $G$ on $n\geq3$ vertices with minimum degree at least $n/2$ contains a Hamilton cycle. The \emph{$k$th power} of a Hamilton cycle is a widely studied strengthening of Hamiltonicity. Formally, it is obtained by joining every pair of vertices whose distance along the Hamilton cycle is at most $k$. The case $k=2$ is called the \emph{square} of a Hamilton cycle. In 1962, P\'{o}sa \cite{Posa} conjectured that every $n$-vertex graph $G$ with $\delta(G)\geq2n/3$ contains the square of a Hamilton cycle, thereby extending Dirac's theorem. Seymour \cite{Seymour} later proposed the following generalization:
\begin{conjecture}[{\cite{Posa, Seymour}}]\label{c}
Let $n$ and $k$ be positive integers with $n\geq k+1$ and $k\geq2$. If $G$ is a graph on $n$ vertices and the minimum degree $\delta(G)\geq\frac{k}{k+1}n$, then $G$ contains the $k$th power of a Hamilton cycle.
\end{conjecture}
Note that the minimum degree condition in Conjecture \ref{c} is tight. Significant progress on the problem emerged only in the $1990$s. Following earlier partial results towards the P\'{o}sa-Seymour conjecture, see \cite{Fan1, Fan2, Fan3, Fan4, Fan5, Komlos, Komlos1}, Koml\'{o}s, S\'{a}rk\"{o}zy, and Szemer\'{e}di \cite{Komlos2} leveraged Szemer\'{e}di's regularity lemma to establish Conjecture \ref{c} for all sufficiently large graphs. Subsequently, the proof of P\'{o}sa's conjecture for large graphs was obtained using more elementary arguments, avoiding the regularity lemma entirely \cite{Chau,Levitt}.

\smallskip

For digraphs, results on powers of Hamilton cycles were scarce until recently. The \emph{minimum semi-degree} $\delta^0(D)$ of a digraph $D$  is the minimum of all the in- and out-degrees of the vertices in $D$. The \emph{minimum total degree} $\delta(D)$ is the minimum number of arcs incident to a vertex in $D$. In addition to extremal embedding questions, the algorithmic detection and property testing of prescribed directed subgraphs have also been studied; see, for example, Alon and Shapira \cite{Alon}. In the setting of tournaments, Bollob\'{a}s and H\"{a}ggkvist first \cite{Bollobas} proved that for fixed $k\in\mathbb{N}$, every $n$-vertex tournament $D$ with $\delta^0(D)\geq(1+o(1))n/4$ contains the $k$th power of a Hamilton cycle. Recently, Dragani\'{c}, Munh\'{a} Correia and Sudakov \cite{Draganic1} showed that one can take $\delta^0(D)\geq n/4+cn^{1-1/\lceil k/2\rceil}$ for some constant $c=c(k)$. Dragani\'{c} et al. \cite{Draganic} proved that every tournament contains the $k$th power of a directed path of length at least $\frac{n}{2^{4k+6}k}$ and a square path of length at least $\lceil2n/3\rceil-1$.

Determining the minimum semi-degree condition for the square of
a Hamilton cycle in general digraphs appears to be significantly more
difficult. It was first posed by Treglown~\cite{Treglown} and later
formulated explicitly by DeBiasio et al.~\cite{Debiasio} as the following
question, which remains completely open even for $k=2$.

\begin{question}[\cite{Debiasio}]\label{ques}
Does every $n$-vertex digraph $D$ with $\delta^0(D)\geq \frac{k}{k+1}n$ contain the $k$th power of a Hamilton cycle?
\end{question}

\smallskip

While the above results focus on minimum semi-degree conditions, a closely related direction concerns the minimum total degree threshold for the same problem. The minimum total degree threshold for the $k$th power of a Hamilton cycle in a digraph was proposed by DeBiasio et al. \cite{Debiasio}.
\begin{conjecture}[\cite{Debiasio}]\label{conj2} Let $k\in\mathbb{N}$ and suppose that $n\in\mathbb{N}$ is sufficiently large. Write $n=(k+3)q+r$, where $q$ is a nonnegative integer and $0\leq r\leq k+2$. Every $n$-vertex digraph $D$ satisfying \begin{align*}
\delta(D)\geq
\begin{cases}
2\lceil(1-\frac{1}{k+3})n\rceil-3, & \text{if }r=k+2,\\
2\lceil(1-\frac{1}{k+3})n\rceil-2, & \text{if }r\in\{k,k+1\},\\
2\lceil(1-\frac{1}{k+3})n\rceil-1, & \text{otherwise}, \end{cases} \end{align*} contains the $k$th power of a Hamilton cycle. \end{conjecture}
If it is true, then the minimum total degree condition is tight, see \cite{Debiasio}. In particular, Ghouila-Houri's theorem \cite{Ghouila} implies that Conjecture \ref{conj2} holds for $k=1$ since an $n$-vertex digraph $D$ with $\delta(D)\geq\lfloor\frac{3n}{2}\rfloor-1$ is strongly connected. For the case of $k=2$, their paper also gave the following result.
\begin{theorem}[\cite{Debiasio}]
For any $\eta>0$, there exists $n_0\in\mathbb{N}$ such that for any $n\geq n_0$, if $D$ is an $n$-vertex digraph with $\delta(D)\geq(\frac{8}{5}+\eta)n$, then $D$ contains the square of a Hamilton cycle.
\end{theorem}
Determining sharp thresholds is a central objective in extremal graph theory, as such results reveal the precise boundary between the existence and non-existence of the desired structure. In this paper, we establish the minimum total degree threshold for the square of a Hamilton cycle in a digraph. Our result improves upon the bound of DeBiasio et al. and confirms Conjecture \ref{conj2} for the case of $k=2$.
\begin{theorem}\label{szc2}
There exists $n_0\in\mathbb{N}$ such that for every $n\ge n_0$,
every $n$-vertex digraph $D$ with
$\delta(D)\ge 8n/5-c$ contains the square of a Hamilton cycle,
where $c=2$ if $n\equiv2,4\pmod 5$, and $c=1$ otherwise.
\end{theorem}

In particularly, the corresponding problem for oriented graphs is much less understood, even
for squares of Hamilton cycles. Treglown~\cite{Treglown} first gave a
construction yielding the asymptotic lower bound $5n/12$. This was
subsequently improved by DeBiasio, who used a slightly unbalanced blow-up of
the Paley tournament on seven vertices to obtain the lower bound $3n/7-1$,
see also~\cite{Debiasio}. More recently, DeBiasio et al.~\cite{Debiasio}
constructed, for every $n\in 11\mathbb N$, an $n$-vertex oriented graph with
minimum semi-degree at least $5n/11-2$ that does not contain the square of a
Hamilton cycle. They also proved that, for every $k\geq2$, every sufficiently
large oriented graph $D$ satisfying
$\delta^0(D)\geq(1/2-1/10^{6000k})n$ contains the $k$th power of a Hamilton
cycle. It remains open to determine the optimal minimum semi-degree threshold,
even asymptotically, for the square of a Hamilton cycle in an oriented graph.

A slightly sharper lower bound  follows directly
from the construction in Proposition~$3.2$ of DeBiasio et
al.~\cite{Debiasio}.

\begin{proposition}\label{prop1}
For every positive integer $n\equiv11\pmod{22}$, there exists an
$n$-vertex oriented graph $G$ such that
$\delta^0(G)=5n/11-1$
and $G$ contains no square of a Hamilton cycle.
\end{proposition}

\cref{prop1} leads naturally to the following  question.

\begin{question}\label{ques:oriented-hamilton-square}
Does every sufficiently large $n$-vertex oriented graph $D$ with
$\delta^0(D)\geq \left\lceil \frac{5n}{11}\right\rceil$
contain the square of a Hamilton cycle?
\end{question}

\smallskip

\noindent \textbf{Organization.} In Section \ref{Section2}, we first introduce the necessary notation, then prove Proposition \ref{prop1}, and finally outline the proof of Theorem \ref{szc2}. Section \ref{Section3} establishes the almost covering, reservoir, and connecting lemmas, and uses them to construct an almost spanning square cycle in the non-extremal case. In Section \ref{Section4}, we prove the extremal theorem, treating the $\beta$-stable and non-$\beta$-stable cases separately. Section \ref{Section5} combines these ingredients through a unified completion argument to prove Theorem \ref{szc2}. Finally, Section \ref{Section6} records two related open questions.

\section{Preparations for Theorem \ref{szc2}}\label{Section2}
\subsection{Notation}
Throughout the paper, $\mathbb N$ denotes the set of positive integers, and $[t]:=\{1,\ldots,t\}$ for every $t\in\mathbb N$. We follow standard digraph notation; see, for example, \cite{Bang-Jensen3}. All digraphs are finite and loopless, but both arcs $uv$ and $vu$ may be present.

\smallskip

Let $D$ be a digraph. We write $V(D)$ and $A(D)$ for its vertex and arc sets, respectively, and set $|D|:=|V(D)|$ and $e(D):=|A(D)|$. The notation $uv$ denotes the arc directed from $u$ to $v$. For $v\in V(D)$, let $N_D^+(v)$ and $N_D^-(v)$ be its out- and in-neighborhoods, and let $d_D^+(v):=|N_D^+(v)|$ and $d_D^-(v):=|N_D^-(v)|$. We write $d_D(v):=d_D^+(v)+d_D^-(v)$, $\delta(D):=\min_{v\in V(D)}d_D(v)$ and $\delta^0(D):=\min_{v\in V(D)}\min\{d_D^+(v),d_D^-(v)\}$. For $X\subseteq V(D)$ and $\sigma\in\{+,-\}$, we define  $N_D^\sigma(v,X):=N_D^\sigma(v)\cap X$, $d_D^\sigma(v,X):=|N_D^\sigma(v,X)|$, and $d_D(v,X):=d_D^+(v,X)+d_D^-(v,X)$. When $D$ is clear, we also write $N_X^\sigma(v)$, $d_X^\sigma(v)$, and $d_X(v)$ for these quantities.

\smallskip

For $X,Y\subseteq V(D)$, let $E_D(X,Y):=\{xy\in A(D):x\in X,\ y\in Y\}$ and $e_D(X,Y):=|E_D(X,Y)|$. We write $D[X]$ for the subdigraph induced by $X$, and $D-X=D\setminus X:=D[V(D)\setminus X]$. If $H$ is a subdigraph of $D$, then $D-H$ means $D-V(H)$. Subscripts are omitted whenever the underlying digraph is clear.

\smallskip

All paths and cycles are directed, and \emph{disjoint} means vertex-disjoint. A sequence of distinct vertices $v_1,\ldots,v_\ell$ is a \emph{$k$-path} if $v_iv_j\in A(D)$ whenever $1\le j-i\le k$. A cyclic ordering $v_1\ldots v_\ell v_1$ is a \emph{$k$-cycle} if $v_iv_{i+s}\in A(D)$ for every $i\in[\ell]$ and $s\in[k]$, with indices taken modulo $\ell$. We call $2$-paths and $2$-cycles \emph{square paths} and \emph{square cycles}, respectively, and write $C_l^k$ for the $k$th power of a directed cycle on $l$ vertices. For a
$2$-path $P=v_1\cdots v_t$ with $t\ge4$, its initial and terminal arcs are $v_1v_2$ and $v_{t-1}v_t$, and $V_{\rm int}(P):=V(P)\setminus\{v_1,v_2,v_{t-1},v_t\}.$ The order of a path or cycle is its number of vertices.

\smallskip

Throughout the paper, we omit all floor and ceiling signs whenever these are not crucial. We use standard hierarchy notation, that is, we write $0<\alpha\ll\beta\ll\gamma$ to mean that we can choose the constants $\alpha, \beta, \gamma$ from right to left. More precisely, there are increasing functions $f$ and $g$ such that, given $\gamma$, whenever we choose $\beta\leq f(\gamma)$ and $\alpha\leq g(\beta)$, all calculations needed in our proof are valid. For the real numbers $a$ and $b$, we use $a\pm b$ to represent an unspecified real number in the interval $[a-b, a+b]$.

\subsection{Proof of Proposition \ref{prop1}}
\begin{proof}[\textbf{Proof of Proposition \ref{prop1}}]
Set $t:=n/11$, so that $t$ is odd. Partition $V(G)$ into
$V_1,V_2,V_3$ with $|V_1|=3t$ and $|V_2|=|V_3|=4t$, and orient all edges
between these sets cyclically as $V_1\to V_2\to V_3\to V_1$. For each
$i\in\{2,3\}$, obtain $G[V_i]$ by deleting one vertex from a regular
tournament of order $4t+1$. Thus every vertex of $G[V_i]$ has internal
in- and outdegree in $\{2t-1,2t\}$. Finally, partition $V_1$ into three
independent sets $A_1,A_2,A_3$ of size $t$ and orient all edges between
them as $A_1\to A_2\to A_3\to A_1$.

Every vertex of $V_1$ has in- and out-degree $5t$. Moreover, for
$v\in V_2$ and $w\in V_3$,
$$
\begin{aligned}
d^-(v)&\geq 3t+(2t-1)=5t-1,
& d^+(v)&\geq 4t+(2t-1)=6t-1,\\
d^+(w)&\geq 3t+(2t-1)=5t-1,
& d^-(w)&\geq 4t+(2t-1)=6t-1.
\end{aligned}
$$
After a vertex is deleted from a regular tournament of order $4t+1$,
exactly $2t$ of the remaining vertices have internal indegree $2t-1$, and
exactly $2t$ have internal outdegree $2t-1$. Hence the lower bound is
attained, and therefore $\delta^0(G)=5t-1=5n/11-1$.

It remains to show that $G$ contains no square of a Hamilton cycle. Observe
that $G[V_1]$ contains no transitive triangle: three vertices meeting all
of $A_1,A_2,A_3$ form a directed triangle, while any three vertices with
two in the same $A_i$ contain a non-adjacent pair.

Suppose to the contrary that $G$ contains the square of a Hamilton cycle
$C$, and consider a maximal interval of consecutive vertices of $C$ lying
in $V_1$. Its predecessor lies in $V_3$ and its successor lies in $V_2$.
The interval cannot consist of a single vertex, since the square of $C$
would then require an edge from $V_3$ to $V_2$, whereas all such edges are
directed from $V_2$ to $V_3$. Nor can it contain at least three vertices,
since its first three vertices would form a transitive triangle in
$G[V_1]$. Thus every maximal $V_1$-interval contains exactly two vertices.

Consequently, $|V_1|$ is even, contradicting $|V_1|=3t$ and the fact that
$t$ is odd. This completes the proof of Proposition \ref{prop1}.
\end{proof}

\subsection{Overview of the proof of Theorem~\ref{szc2}}\label{subsection2.2}

For each $n\in\mathbb{N}$, define
\[
c:=
\begin{cases}
2, & \text{if } n\equiv 2,4\pmod 5,\\
1, & \text{otherwise},
\end{cases}
\qquad\text{and}\qquad
\tau(n):=\left\lceil\frac{8n}{5}-c\right\rceil.
\]
Since $\delta(D)$ is integer-valued, the degree condition in
Theorem~\ref{szc2} is equivalent to $\delta(D)\ge \tau(n)$. We shall
use this formulation throughout the paper, together with the estimates $\tau(n)\ge \frac{8n}{5}-2$ and $\tau(n)-(n-1)\ge \frac{3n}{5}-1.$

\smallskip

Fix constants satisfying $0<\varepsilon\ll\gamma\ll\beta\ll\xi<\frac15.$ We say that an $n$-vertex digraph $D$ is
\emph{$(1/5,\gamma)$-extremal} if there exists a set
$S\subseteq V(D)$ such that $\left||S|-\frac n5\right|<\gamma n$ and $e(D[S])<\gamma n^2.$ The proof is divided into the extremal and non-extremal cases. In
each case, we construct a $2$-cycle covering all but at most $\xi n$
vertices of $D$. Since $\xi<1/5$, the maximality argument in
Section~\ref{Section5} then implies that a maximum-order $2$-cycle
is spanning.

\medskip

\noindent\textbf{\large The non-extremal case.}
Suppose that $D$ is not $(1/5,\gamma)$-extremal. We combine the
reservoir, connecting, and almost covering lemmas in three steps.

\smallskip

\noindent\textbf{(N1) Constructing a reservoir.}
We choose a small set $\mathcal R\subseteq V(D)$ such that any two
disjoint arcs in $D-\mathcal R$ can be joined by a short $2$-path
whose internal vertices lie in $\mathcal R$. Moreover, the connecting
path can be chosen while avoiding any bounded set of previously used
vertices of $\mathcal R$.

\smallskip

\noindent\textbf{(N2) Finding an almost-spanning path cover.}
In $D-\mathcal R$, we find a bounded family of pairwise disjoint
$2$-paths covering all but at most $\varepsilon n$ vertices.
Discarding only $O(1)$ vertices in total, we may assume that the
initial and terminal arcs of each path are disjoint.

\smallskip

\noindent\textbf{(N3) Joining the paths into a cycle.}
We place the paths in a cyclic order and join consecutive paths using
internally disjoint connecting paths through $\mathcal R$. Joining
the last path back to the first produces a single $2$-cycle, which
leaves at most $\xi n$ vertices of $D$ uncovered.

\medskip

\noindent\textbf{\large The extremal case.}
Suppose that $D$ is $(1/5,\gamma)$-extremal. The proof again proceeds
in three steps.

\smallskip

\noindent\textbf{(E1) Cleaning the extremal partition.}
Let $S\subseteq V(D)$ witness extremality and put $D^*:=D-S.$ After deleting a small set of atypical vertices from $S$, we obtain
a set $A\subseteq S$ such that every vertex of $A$ forms a digon with
all but a small number of vertices of $D^*$. The vertices of $A$ will
be used to connect paths and, subsequently, to extend the resulting
cycle.

\smallskip

\noindent\textbf{(E2) Constructing a core $2$-cycle.}
We distinguish two cases according to whether $D^*$ is $\beta$-stable. Let $\overline{\mathcal K}_4$ denote the family of loopless multigraphs on four vertices in which every pair of distinct vertices has multiplicity one or two, and the pairs of multiplicity one form a possibly empty matching. This particular family arises naturally in the $\beta$-stable case. Indeed, as we shall see in Lemma~\ref{lem:noextremal tiling}, every sufficiently dense standard multigraph admits an almost-perfect tiling by members of $\overline{\mathcal K}_4$.

If $D^*$ is $\beta$-stable, the diregularity lemma and an
almost-perfect $\overline{\mathcal K}_4$-tiling of the associated reduced multigraph yield a bounded family of long $2$-paths covering almost all vertices of $D^*$. These paths are then joined cyclically using short connecting paths.

If $D^*$ is not $\beta$-stable, then $D^*$ has an approximate
two-block structure with two internally dense sets $L_0$ and $R_0$,
and almost all arcs between them are directed from $R_0$ to $L_0$.
We construct short $2$-path connections in both directions and
combine them with almost spanning $2$-paths in $D[L_0]$ and
$D[R_0]$.

In either subcase, we obtain a core $2$-cycle covering almost all
vertices of $D^*$ and using only a small number of vertices from $A$.

\smallskip

\noindent\textbf{(E3) Inserting the remaining vertices of $A$.}
We insert the unused vertices of $A$ at pairwise well-separated
positions on the core $2$-cycle. Since both $S\setminus A$ and the
set of uncovered vertices of $D^*$ are small, the resulting
$2$-cycle leaves at most $\xi n$ vertices of $D$ uncovered.

Thus, in both cases, $D$ contains a $2$-cycle covering all but at
most $\xi n$ vertices. The maximality argument from
Section~\ref{Section5} now completes the proof of
Theorem~\ref{szc2}.

\section{The non-extremal case of Theorem \ref{szc2}}\label{Section3}
As noted in Subsection \ref{subsection2.2}, the non-extremal argument uses an almost covering lemma, a reservoir lemma, and a connecting lemma. We first state the covering and reservoir lemmas. The reservoir lemma will be proved after the connecting lemma, from which it follows.


\subsection{Almost covering and connecting lemmas}

The lemma below serves as our covering lemma. In fact, the almost
covering lemma requires a weaker minimum total degree condition than
that in Theorem~\ref{szc2}.

\begin{lemma} [{\cite{Debiasio}}, Almost covering lemma] \label{cover1}
Given any integer $k\geq2$
and $\alpha>0$, there exist $n_0, T\in\mathbb{N}$ such that for any $n\geq n_0$ the following holds. If $D$ is an $n$-vertex digraph with $\delta(D)\geq 2(1-1/(k+2)+\alpha)n$, then $D$ contains a collection of at most $T$ disjoint $k$-paths that covers all but at most $\alpha n$ vertices.
\end{lemma}

For an ordered pair of disjoint arcs $uv,yz\in A(D)$, we call $(uv,yz)$ \emph{direct} if $uvyz$ itself is a $2$-path.

\begin{definition}\label{def:direct-pair}
Let $D$ be a digraph. A set $\mathcal R\subseteq V(D)$ is a \emph{$(T,C)$-reservoir} if, for every ordered pair of disjoint arcs $uv,yz\in A(D-\mathcal R)$ and every set $U\subseteq\mathcal R$ with $|U|\leq TC$, there exists a
$2$-path from $uv$ to $yz$ of order at most $C$ whose internal vertices lie in $\mathcal R\setminus U$.
\end{definition}

\begin{lemma}[Reservoir lemma]\label{reservoir-lemma}
Let $0<1/n\ll1/T\ll\lambda\ll\mu\ll\gamma<1/30$, and let $D$ be an $n$-vertex digraph satisfying $\delta(D)\ge(8/5-\gamma/2)n$. Suppose that $D$ is not $(1/5,\gamma)$-extremal. Then $D$ contains a $(T,C_\gamma)$-reservoir $\mathcal R$ with $|\mathcal R|\leq2\lambda n$, where $C_\gamma$ is the constant in Lemma~\ref{connecting-lemma}.
\end{lemma}

The proof of Lemma~\ref{reservoir-lemma} is given after the following connecting lemma, which is the main input in the proof of
Lemma~\ref{reservoir-lemma}.
\begin{lemma}[Connecting Lemma]\label{connecting-lemma}
Let $0<1/n\ll\mu\ll\gamma<1/30$, and let $D$ be an $n$-vertex digraph satisfying
$\delta(D)\ge (8/5-\gamma/2)n$. Suppose that $D$ is not $(1/5,\gamma)$-extremal. Then, for every set $F\subseteq V(D)$ with $|F|\le \mu n$ and every pair of disjoint arcs $uv,yz\in A(D-F)$, there exists a $2$-path $x_1x_2\cdots x_t$ in $D-F$ such that $x_1=u$, $x_2=v$, $x_{t-1}=y$, and $x_t=z$, where
$$
t\le C_\gamma=5\left\lceil\frac{8}{25\gamma}+\frac25\right\rceil+3.
$$
\end{lemma}
\begin{proof}
Choose $\mu$ sufficiently small in terms of $\gamma$ so that
$12\mu<\gamma, 6\mu<\gamma^2$, and $9\mu<1/5-7\gamma/2$.
Let $G$ be the bidirected graph associated with $D$, that is, $V(G)=V(D)$ and $xy\in E(G)$ if and only if both $xy$ and $yx$ belong to $A(D)$. Since each semidegree is at most $n-1$, the degree condition gives, for every $x\in V(D)$,
\begin{equation}\label{eq:conn-degree}
d_D^+(x),\ d_D^-(x),\ d_G(x)
\ge \left(\frac35-\frac\gamma2\right)n+1.
\end{equation}
Consequently, every two vertices $x,y\in V(D)$ satisfy
\begin{equation}\label{eq:conn-common}
|N_G(x)\cap N_G(y)|
\ge \left(\frac15-\gamma\right)n+2.
\end{equation}

We first establish three auxiliary claims.

\begin{claim}\label{cl:conn-fresh}
Let $X,Z\subseteq V(D)$ satisfy $|Z|\le3\mu n$ and
$|X|\ge \left(\frac15-\gamma\right)n-3\mu n$.
Then $D[X\setminus Z]$ contains an arc.
\end{claim}

\begin{proof}\renewcommand*{\qedsymbol}{$\blacksquare$}
Suppose that $D[X\setminus Z]$ contains no arc, and put
$L:=\lceil(1/5-\gamma)n\rceil+1$. If $|X\setminus Z|\ge L$, then any $L$-subset $S$ of $X\setminus Z$ satisfies $e(D[S])=0$ and witnesses that $D$ is $(1/5, \gamma)$-extremal. Otherwise, extend $X\setminus Z$ to an $L$-set $S$. Since
$|X\setminus Z|\ge(1/5-\gamma)n-6\mu n$, at most $6\mu n+2$ vertices are added. Every arc of $D[S]$ is incident with one of the added vertices, and hence, for sufficiently large $n$, $e(D[S])\le2(6\mu n+2)n<\gamma n^2$. Moreover, $\bigl||S|-n/5\bigr|<\gamma n$, contradicting non-extremality.
\end{proof}

\begin{claim}\label{cl:conn-alpha}
Let $\alpha(D)$ denote the largest order of a vertex set spanning no arc. Then
\begin{equation}\label{eq:conn-alpha}
\alpha(D)\le
\left(\frac15-\frac{7\gamma}{2}-\gamma^2\right)n.
\end{equation}
Consequently, if $X,Z\subseteq V(D)$ satisfy $|Z|\le3\mu n$ and
$|X|\ge(1/5-7\gamma/2)n$, then $D[X\setminus Z]$ contains an arc.
\end{claim}

\begin{proof}\renewcommand*{\qedsymbol}{$\blacksquare$}
Let $I$ be an independent set and put $L:=\lceil(1/5-\gamma)n\rceil+1$. If $|I|\ge L$, then an $L$-subset of $I$ witnesses that $D$ is $(1/5, \gamma)$-extremal. Otherwise, extend $I$ to an $L$-set $S$. Every arc of $D[S]$ has an end-vertex in $S\setminus I$, so non-extremality gives $\gamma n^2\le e(D[S])\le2L(L-|I|)$.
It follows that
$$
|I|\le L-\frac{\gamma n^2}{2L}
=\left(\frac15-\gamma-\frac{\gamma}{2(1/5-\gamma)}+o(1)\right)n.
$$
Since
$
\frac{\gamma}{2(1/5-\gamma)}-\frac{5\gamma}{2}
=\frac{5\gamma^2}{2(1/5-\gamma)}>\gamma^2,
$
\eqref{eq:conn-alpha} follows for sufficiently large $n$. The final assertion follows from $6\mu<\gamma^2$.
\end{proof}

For a digraph $H$, define an \emph{auxiliary digraph} $\mathcal{A}_2(H)$ with vertex set $A(H)$, in which $ab\to bc$ in $\mathcal A_2(H)$ if and only if $ac\in A(H)$. Thus, a sequence $v_1\ldots v_\ell$ is a $2$-walk in $H$ if and only if $v_1v_2\to v_2v_3\to\cdots\to v_{\ell-1}v_\ell$ is a directed walk in $\mathcal{A}_2(H)$.

\begin{claim}\label{cl:conn-state}
For every $Z\subseteq V(D)$ with $|Z|\le3\mu n$, the auxiliary digraph $\mathcal{A}_2(D-Z)$ is strongly connected.
\end{claim}

\begin{proof}\renewcommand*{\qedsymbol}{$\blacksquare$}
Write $z:=|Z|$. By \eqref{eq:conn-degree} and \eqref{eq:conn-common}, we have $\delta(G-Z)\geq(3/5-\gamma/2)n+1-z$, every two vertices of $G-Z$ have at least $(1/5-\gamma)n+2-z$ common neighbors, and every vertex has at most $(2/5+\gamma/2)n-2$ non-neighbors in $G-Z$.

\smallskip

We first show that every strong component of $\mathcal{A}_2(D-Z)$ contains an arc belonging to a digon of $D-Z$. Let $xy$ be an arc belonging to a digon. By Claim~\ref{cl:conn-fresh}, $D[N_{G-Z}(x)\cap N_{G-Z}(y)]$ contains an arc $pq$. The $2$-walk $xypqyx$ shows that $xy$ reaches $yx$ in $\mathcal{A}_2(D-Z)$. Interchanging $x$ and $y$, we see that $yx$ also reaches $xy$. Thus the two arcs of every digon lie in the same strong component of $\mathcal{A}_2(D-Z)$. If $x,y,w$ span a copy of $\overleftrightarrow{K_3}$ in $D-Z$,  then $xy\to yw\to wx\to xy$ and $yx\to xw\to wy\to yx$ are directed cycles in $\mathcal{A}_2(D-Z)$. Since $xy$ and $yx$ lie in the same strong component, all six arcs of this copy of $\overleftrightarrow{K_3}$ lie in that component.

\smallskip

Now let $ab\in A(D-Z)$. Choose an arc $pq$ in $D[N_{G-Z}(a)\cap [N_{G-Z}(b)]$. The $2$-walk $abpqab$ shows that $ab$ lies on a closed directed walk in $\mathcal{A}_2(D-Z)$ containing the arcs $bp$ and $qa$, each of which belongs to a digon of $D-Z$. Hence the strong component containing $ab$ contains an arc belonging to a digon. Therefore every strong component of $\mathcal{A}_2(D-Z)$ contains such an arc.

\smallskip

Suppose that $\mathcal{A}_2(D-Z)$ is not strongly connected. In its condensation choose a source component. Let $L$ be its vertex set and let $R:=V(\mathcal{A}_2(D-Z))\setminus L$. Then $R\cup L$ is nontrivial, each side is a union of strong components, and there is no arc of $\mathcal{A}_2(D-Z)$ directed from $R$ to $L$. Since the two arcs of every digon lie in the same strong component, label each edge $xy\in E(G-Z)$ by $R$ or $L$, according to which of these two sets contains the arcs $xy$ and $yx$.

\smallskip

Assume that a vertex $v$ is incident with edges of $G-Z$ carrying both labels, and call it \emph{mixed}. Define $A:=\{x\in N_{G-Z}(v):vx\in R\}$ and $C:=\{x\in N_{G-Z}(v):vx\in L\}$, and put $B:=V(D-Z)\setminus(A\cup C\cup\{v\})$. Every neighbor of $v$ in $G-Z$ lies in exactly one of $A$ and $C$, so $N_{G-Z}(v)=A\cup C$ and $|B|=(n-z-1)-d_{G-Z}(v)\leq(2/5+\gamma/2)n-2$.

\smallskip

If $a\in A$ and $c\in C$, then $ac\notin E(G-Z)$, since otherwise, $v,a,c$ would span a copy of $\overleftrightarrow{K_3}$, placing $va$ and $vc$ in the same strong component of $\mathcal{A}_2(D-Z)$. More strongly, there is no arc from $A$ to $C$. Otherwise, if $ac\in A(D)$, since $va,av,vc,cv\in A(D-Z)$, we obtain the directed walk
$va\to ac\to cv$ in $\mathcal{A}_2(D-Z)$ from $R$ to $L$, a contradiction. Here $cv\in L$ because $cv$ and $vc$ are the two arcs of the same digon.

\smallskip

If $a\in A$ and $x\in N_{G-Z}(v)\cap N_{G-Z}(a)$, then $v,a,x$ span a copy of $\overleftrightarrow{K_3}$ in $D-Z$, so $vx$ lies in the same strong component as $va$ and $x\in A$. The analogous statement holds for $C$. Hence
\begin{equation}\label{eq:conn-cut-sizes}
|A|,|C|\geq(1/5-\gamma)n+2-z,
\qquad |B|\leq(2/5+\gamma/2)n-2.
\end{equation}

In the following, we set $h:=(1/5-\gamma)n+2-z$. Fix $a_0\in A$ and $c_0\in C$. Their common neighborhood in $G-Z$ has size at least $h$. After deleting $v$, the remaining set still has size at least $(1/5-\gamma)n+1-z\geq(1/5-\gamma)n-3\mu n$, so Claim~\ref{cl:conn-fresh} gives an arc $pq$ whose end-vertices lie in this common neighborhood with $p,q\ne v$. If $pq\in R$, then $pq\to qc_0$; since no arc of $\mathcal{A}_2(D-Z)$ goes from $R$ to $L$, we have $qc_0\in R$; hence $c_0$ is mixed. If $pq\in L$, then $a_0p\to pq$ implies $a_0p\in L$; hence $a_0$ is mixed. Thus for every $(a_0,c_0)\in A\times C$, at least one endpoint is mixed. Consequently, either every vertex of $A$ is mixed, or every vertex of $C$ is mixed.

\smallskip

Assume first that every vertex of $A$ is mixed. For $x\in A$, define $E_x:=\{w\in N_{G-Z}(x): xw\in L\}$. Since $x$ is mixed, $E_x\neq\emptyset$. Every edge $aa^\prime\in E(G[A])$ is labelled $R$, since $v, a, a^\prime$ span a copy of $\overleftrightarrow{K_3}$, and $E_G(A, C)=\emptyset$; hence $E_x\subseteq B$. Choose $w_x\in E_x$. Every common neighbor of $x$ and $w_x$ in $G-Z$ also lies in $E_x$, because these three vertices span a copy of $\overleftrightarrow{K_3}$, and all of whose arcs lie in the same strong component of $\mathcal{A}_2(D-Z)$ as the arc $xw_x\in L$. Therefore $|E_x|\geq h$. Also, every common neighbor of $v$ and $x$ in $G-Z$ lies in $A$, so $d_{G-Z}(x,A)\geq h$.

By \eqref{eq:conn-cut-sizes}, we have that
\begin{equation}\label{eq:conn-positive}
3h-|B|\geq(1/5-7\gamma/2)n+8-3z>0.
\end{equation}
Fix $x\in A$ and put $P:=N_{G-Z}(x,A)$. Claim~\ref{cl:conn-fresh} gives an arc $p_1p_2$ in $D[P]$. Since $vp_1\in R$ and $vp_1\to p_1p_2$, we have $p_1p_2\in R$. For $i=1,2$, we have $E_{p_i}\cap E_x=\emptyset$: otherwise a vertex in the intersection would yield a copy of $\overleftrightarrow{K_3}$ containing the arcs $xp_i\in R$ and $xw_x\in L$, contradicting that $R$ and $L$ are disjoint unions of strong components.  Hence $|E_{p_1}\cap E_{p_2}|\geq |E_{p_1}|+|E_{p_2}|-|B\setminus E_x|\geq3h-|B|>0$. For $w\in E_{p_1}\cap E_{p_2}$, the arc $p_1p_2\to p_2w$ of $\mathcal{A}_2(D-Z)$ goes from $R$ to $L$, a contradiction.

If every vertex of $C$ is mixed, define $F_x:=\{w\in N_{G-Z}(x): xw\in R\}$. The same argument gives $F_x\subseteq B$, $|F_x|\geq h$, and $d_{G-Z}(x,C)\geq h$. Choose an arc $p_1p_2$ in $D[N_{G-Z}(x,C)]$. Since $p_1p_2\to p_2v$ and $p_2v\in L$, we have $p_1p_2\in L$. Inequality \eqref{eq:conn-positive} gives $|F_{p_1}\cap F_{p_2}|>0$, and any $w$ in this intersection yields an arc $wp_1\to p_1p_2$ of $\mathcal{A}_2(D-Z)$ directed from $R$ to $L$. Thus no mixed vertex exists.

Finally, $\delta(G-Z)>(n-z)/2,$ so $G-Z$ is connected. Consider the above $L/R$-labelling of $E(G-Z)$. Since no vertex is mixed, all edges of $G-Z$ incident with a fixed vertex receive the same label. As $G-Z$ is connected, every edge of $G-Z$ receives the same label. But every strong component of $\mathcal{A}_2(D-Z)$ contains an arc belonging to a digon of $D-Z$, whereas both $R$ and $L$ are nonempty unions of strong components. This contradiction proves that $\mathcal{A}_2(D-Z)$ is strongly connected.
\end{proof}

For an arc $e=ab$, define $P_e:=N^+(a)\cap N^+(b),  Q_e:=N^-(a)\cap N^-(b),$ and
$p_e:=|P_e|$ and  $q_e:=|Q_e|$.
By \eqref{eq:conn-degree}, we have that $p_e,q_e\ge(1/5-\gamma)n+2$. Moreover, for every arc $e=ab$,
\begin{align}\label{eq:conn-pq-sum}
p_e+q_e
\ge d^+(a)+d^+(b)-n+d^-(a)+d^-(b)-n=d_D(a)+d_D(b)-2n
\ge\left(\frac65-\gamma\right)n.
\end{align}
Suppose that $e=ab$ and $f=cd$ are vertex-disjoint arcs satisfying $p_f\le p_e+5\gamma n/2$. It follows from \eqref{eq:conn-pq-sum} that
\begin{equation*}
|P_e\cap Q_f|
\ge p_e+q_f-n
\ge\left(\frac15-\frac{7\gamma}{2}\right)n.
\end{equation*}
By Claim~\ref{cl:conn-alpha}, if $W\subseteq V(D)$ satisfies
$|W\cup\{a,b,c,d\}|\le3\mu n$, then
$(P_e\cap Q_f)\setminus(W\cup\{a,b,c,d\})$ contains an arc $pq$. Consequently, $abpqcd$ is a six-vertex $2$-path connector from $e$ to $f$ whose internal vertices avoid $W$ (for simplicity, say $abpqcd$ is a $W$-avoiding six-vertex connector from $e$ to $f$).

\smallskip

We construct the required $2$-path iteratively. Start with the $2$-path $P=uv$, and keep the arc $yz$ unused as the prescribed terminal arc. Suppose that the path constructed so far is a $2$-path in $D-F$, ends with the arc $e=ab$, and contains neither $y$ nor $z$. Write its vertex set as $U\cup\{a,b\}$. At the beginning of each iteration, every preceding nonterminal iteration has increased the value of  $p_e$ by more than $5\gamma n/2$.
Since
$(1/5-\gamma)n+2\le p_e\le n-2$, the number of preceding nonterminal iterations is bounded in terms of $\gamma$. Consequently, the path
constructed so far has $O_\gamma(1)$ vertices. Since $1/n\ll\mu$,
every forbidden set used below has size at most $3\mu n$ for
sufficiently large $n$.

\smallskip

Since $p_e,q_{yz}\ge(1/5-\gamma)n+2$, Claim~\ref{cl:conn-fresh} gives an arc $cd\in A(D[P_e])$ avoiding $F\cup U\cup\{a,b,y,z\}$. Applying the claim again, while also excluding $c,d$, gives an arc $rs\in A(D[Q_{yz}])$. Thus $cd$ and $rs$ are disjoint and avoid $F\cup U\cup\{a,b,y,z\}$. Put $Z:=F\cup U\cup\{a,b,y,z\}$. By the preceding observation, $|Z|\le3\mu n$. By Claim~\ref{cl:conn-state}, there is a directed path from $cd$ to $rs$ in $\mathcal{A}_2(D-Z)$. Together with $ab\to bc\to cd$ and $rs\to sy\to yz$, this gives a directed walk $e_0e_1\ldots e_\ell$ in $\mathcal{A}_2(D)$ from $e_0=e$ to $e_\ell=yz$, where $e_1=bc$, $e_2=cd$, and every arc from $e_2$ through $rs$ belongs to $A(D-Z)$.

\smallskip

Let $j$ be the largest index such that $p_{e_j}\le p_{e_0}+5\gamma n/2$. If $j=\ell$, use a $(F\cup U)$-avoiding six-vertex connector from $e_0$ to $yz$, and finish.

\smallskip

Suppose that $j<\ell$. If $j=0$, extend the current $2$-path by $c$, the head of $e_1$. If $j=1$, extend the current $2$-path by $c,d$. Now suppose that $j\ge2$, and write $e_j=xw$ and $e_{j+1}=wh$. The arcs $e_0$ and $e_j$ are vertex-disjoint, and $e_j$ avoids the path constructed so far. Use a $(F\cup U\cup\{y,z,h\})$-avoiding six-vertex connector from $e_0$ to $e_j$, and then extend the resulting $2$-path by $h$, using the arc $e_j\to e_{j+1}$ of $\mathcal{A}_2(D)$. Since $wh,xh\in A(D)$ and $D$ is loopless, $h\notin\{x,w\}$. Moreover, $h\notin V(P)$ by the choice of the directed walk. Thus the resulting sequence remains a $2$-path in $D-F$. In each case at most five vertices are added, and, by the maximality of $j$, the new terminal arc satisfies $p_{e_{j+1}}>p_{e_0}+5\gamma n/2$.

\smallskip

If $e_{j+1}=yz$, the construction is complete. If $e_j=rs$, then $e_{j+1}=sy$; in this case extend the path by $z$, using the arc $sy\to yz$ of $\mathcal{A}_2(D)$, and terminate the construction. This terminal iteration adds at most six vertices. In every other case, $e_{j+1}$ avoids $y,z$, and we proceed to the next iteration.

Every nonterminal iteration increases $p_e$ by more than $5\gamma n/2$. Moreover, throughout the construction, $(1/5-\gamma)n+2\leq p_e\leq n-2$. Hence, if $q$ denotes the number of nonterminal iterations, then
$q<\frac{(4/5+\gamma)n-4}{5\gamma n/2}<\frac{8}{25\gamma}+\frac25$.
Let
$K_\gamma:=\left\lceil 8/(25\gamma)+2/5\right\rceil$.
Then $q\leq K_\gamma-1$, so the procedure performs at most
$K_\gamma$ iterations. Every nonterminal iteration adds at most five vertices, while the terminal iteration adds at most six. Therefore
$$
t\leq2+5(K_\gamma-1)+6=5K_\gamma+3=C_\gamma.
$$
Finally, throughout the construction every auxiliary forbidden set consists of $F$, at most $C_\gamma$ path vertices, and a bounded number of temporarily specified vertices. Since $1/n\ll\mu$, all such sets have size at most $3\mu n$ for sufficiently large $n$. This completes the proof.
\end{proof}

\medskip
\noindent\textbf{Comment on the use of AI.}
In an earlier version of this paper, we proved Lemma~\ref{connecting-lemma} by adapting the approach of DeBiasio et al.~\cite{Debiasio}, and obtained a connecting
$2$-path of order at most $42$. That proof, however, involved a substantially more intricate case analysis and was several pages longer than the present one. With the assistance of OpenAI's ChatGPT 5.6, we reorganized and simplified this part of the argument into the form given above. Although the resulting bound $C_\gamma$ is weaker than the previous bound of $42$, it is fully sufficient for all subsequent applications. The remainder of the proof of Theorem \ref{szc2} was developed entirely by the authors.

\medskip

We conclude this section by proving Lemma \ref{reservoir-lemma}.

\begin{proof}[\textbf{Proof of Lemma~\ref{reservoir-lemma}}]
Choose a constant $\rho>0$ sufficiently small that $C_\gamma\rho<\mu$. Fix a non-direct ordered pair of disjoint arcs $(uv,yz)$, and put $m_0:=\lfloor\rho n\rfloor$. We recursively construct $2$-paths $Q_1,\ldots,Q_{m_0}$ from $uv$ to $yz$ whose internal vertex sets are pairwise disjoint. Suppose that $Q_1,\ldots,Q_{j-1}$ have already been chosen, and set $F_j:=\bigcup_{h<j}V_{\rm int}(Q_h).$ Since every $Q_h$ has order at most $C_\gamma$, we have that
\[
|F_j|\leq(C_\gamma-4)(j-1)\leq C_\gamma\rho n<\mu n.
\]
The set $F_j$ is disjoint from $\{u,v,y,z\}$, so Lemma~\ref{connecting-lemma} gives a $2$-path $Q_j$ from $uv$ to $yz$ of order at most $C_\gamma$ whose internal vertices avoid $F_j$. Since $(uv,yz)$ is non-direct, every such path has at least one internal vertex. Thus the paths $Q_1,\ldots,Q_{m_0}$ are distinct and their internal vertex sets are pairwise disjoint. Denote this family by $\mathcal Q(uv,yz)$.

Choose $\mathcal R\subseteq V(D)$ by including every vertex independently with probability $\lambda$. For a fixed non-direct pair $(uv,yz)$ and $Q\in\mathcal Q(uv,yz)$, let $I_Q$ be the indicator of the event $V_{\rm int}(Q)\subseteq\mathcal R$. The variables $I_Q$ are independent, and $\mathbb P(I_Q=1)=\lambda^{|V_{\rm int}(Q)|}\geq\lambda^{C_\gamma-4}.$
For sufficiently large $n$, $m_0\geq\rho n/2$. Hence, if $Y_{uv,yz}:=\sum_{Q\in\mathcal Q(uv,yz)}I_Q$, then $\mathbb EY_{uv,yz}\geq\frac{\rho\lambda^{C_\gamma-4}}2n.$ Put $\kappa:=\rho\lambda^{C_\gamma-4}/4$. Chernoff's inequality gives
\[
\mathbb P(Y_{uv,yz}<\kappa n)\leq\exp(-\Omega(n)).
\]
There are fewer than $n^4$ ordered pairs of disjoint arcs. Moreover, another application of Chernoff's inequality gives $\mathbb P(|\mathcal R|>2\lambda n)\leq\exp(-\Omega(n))$. Therefore, by the union bound, there is a choice of $\mathcal R$ such that $|\mathcal R|\leq2\lambda n$ and, for every non-direct ordered pair $(uv,yz)$, at least $\kappa n$ members of $\mathcal Q(uv,yz)$ have all their internal vertices in $\mathcal R$.

Let $uv,yz\in A(D-\mathcal R)$ be disjoint arcs, and let $U\subseteq\mathcal R$ satisfy $|U|\leq TC_\gamma$. If $(uv,yz)$ is direct, then $u,v,y,z$ is the required $2$-path and has no internal vertex. Otherwise, at least $\kappa n$ members of $\mathcal Q(uv,yz)$ have all their internal vertices in $\mathcal R$. Since these internal vertex sets are pairwise disjoint, at most $|U|$ of those paths meet $U$. By the choice of $n$, we have $\kappa n>TC_\gamma\geq|U|$, so one of them avoids $U$. This path has order at most $C_\gamma$ and all its internal vertices lie in $\mathcal R\setminus U$. Hence $\mathcal R$ is a $(T,C_\gamma)$-reservoir.
\end{proof}
\subsection{An almost spanning $2$-cycle}

\begin{lemma}
\label{nonextremal-cycle}
Suppose that $0<1/n_0\ll\gamma\ll\xi<1/5$. Let $D$ be a digraph with $n\geq n_0$ vertices and $\delta(D)\geq\tau(n)$. If $D$ is not $(1/5,\gamma)$-extremal, then $D$ contains a
$2$-cycle covering all but at most $\xi n$ vertices.
\end{lemma}

\begin{proof}
Choose constants $0<\lambda\ll\mu\ll\gamma$ such that $4\lambda<\xi$ and $\lambda<1/100$. Apply Lemma~\ref{cover1} with $k=2$ and $\alpha=\lambda$, and let $T$ be the resulting bound. Increasing $T$ if necessary, we may assume that $1/T\ll\lambda$. Since $\tau(n)\geq8n/5-2$, for sufficiently large $n$ we have $\delta(D)\geq(8/5-\gamma/2)n$. Lemma~\ref{reservoir-lemma} therefore gives a $(T,C_\gamma)$-reservoir $\mathcal R$ satisfying $|\mathcal R|\leq2\lambda n$.

Put $D_0:=D-\mathcal R$. Then since $\lambda<1/100$ and $n$ is sufficiently large, we have that
\[
\delta(D_0)\geq\tau(n)-2|\mathcal R|\geq\frac85n-2-4\lambda n\quad \mbox{and}\quad \frac85n-2-4\lambda n\geq\left(\frac32+2\lambda\right)n\geq\left(\frac32+2\lambda\right)|D_0|.
\]
Hence Lemma~\ref{cover1} yields at most $T$ disjoint $2$-paths in $D_0$ covering all but at most $\lambda|D_0|\leq\lambda n$ vertices. Delete every one of these paths having order at most $3$. This deletes at most $3T$ additional vertices, and we may assume that $3T\leq\lambda n$. Let $P_1,\ldots,P_s$ be the remaining paths. For sufficiently large $n$, we have $s\geq1$, and every $P_i$ has disjoint initial and terminal arcs. The paths $P_1,\ldots,P_s$ cover all but at most $2\lambda n$ vertices of $D_0$.

List the paths cyclically, and write $a_ib_i$ and $c_id_i$ for the initial and terminal arcs of $P_i$, respectively, with indices taken modulo $s$. For each $i\in[s]$, join $c_id_i$ to $a_{i+1}b_{i+1}$ by the reservoir property. More precisely, before choosing the $i$th connector, let $U$ be the set of reservoir vertices used internally by the preceding connectors. Since there are at most $T$ connectors and each has order at most $C_\gamma$, we always have $|U|\leq TC_\gamma$. Thus the connectors can be chosen so that their internal vertex sets are pairwise disjoint and lie in $\mathcal R$. Concatenating the paths $P_i$ with these connectors, with each common terminal and initial arc identified, gives a $2$-cycle $C_0$.

The vertices not covered by $C_0$ consist of at most $2\lambda n$ vertices of $D_0$ and some unused vertices of $\mathcal R$. Consequently,
\[
|V(D)\setminus V(C_0)|\leq2\lambda n+|\mathcal R|\leq4\lambda n<\xi n.
\]
This proves the lemma.
\end{proof}

\section{The extremal case of Theorem \ref{szc2}}\label{Section4}
In this section, we present several technical tools that will be used to handle the extremal case in the proof of Theorem \ref{szc2}. The central result of this section is the following theorem.
\begin{theorem}\label{acycle}
Suppose that $0<1/n_0\ll\gamma\ll\xi\ll1/5$. Let $D$ be a digraph with $n\ge n_0$ vertices and $\delta (D) \geq\tau(n)$. If $D$ is $(1/5, \gamma)$-extremal, then $D$ contains the square of a cycle $C$ that covers all but at most $\xi n$ vertices.
\end{theorem}
To prove Theorem \ref{acycle}, we first introduce a stability notion that will be used repeatedly.
\begin{definition} [$\beta$-stable digraph] \label{stable}
Let $H$ be a digraph of order $t$. We say that $H$ is \emph{non-$\beta$-stable} if there exist two (not necessarily disjoint) vertex sets $U_1, U_2\subseteq V(H)$ with $|U_i|\geq(1/2-\beta)t$ for every $i\in[2]$ such that $e_H(U_1,U_2)\leq(\beta t)^2$. Otherwise, we say that $H$ is \emph{$\beta$-stable}.
\end{definition}
Building on this notion of stability, we introduce the following concepts of ``almost complete'' structures, which will help us describe the local density of subgraphs.
\begin{definition} [$\beta$-almost complete subdigraphs and $\beta^{1/2}$-almost one-way complete bipartite subdigraphs]
Let $H$ be a digraph of order $t$, and let $V_1, V_2$ be two disjoint vertex subsets of $V(H)$. We say $H[V_1]$ is \emph{$\beta$-almost complete} if there exists a subset $V_1^\prime\subseteq V_1$ with $|V_1^\prime|\leq\beta t$ such that for every vertex $u$ in $V_1\setminus V_1^\prime$, we have $d_{H[V_1]}(u)\geq2(|V_1|-\beta t)$, and for every vertex $v\in V_1$, we have $d_{H[V_1]}(v)\geq|V_1|-\beta t$. Furthermore, we say $(V_1, V_2)$ forms a \emph{$\beta^{1/2}$-almost one-way complete} bipartite subdigraph if $|V_1|=|V_2|\pm4\beta t$, and $e_H(V_1,V_2)\geq|V_1|\cdot|V_2|-(\beta t)^2$.
\end{definition}

The rest of the section is organized as follows. Subsection \ref{subsection4.1} collects the regularity and embedding tools needed in the proof. In Subsection \ref{subsection4.2} we carry out the preprocessing common to both extremal configurations. The $\beta$-stable and non-$\beta$-stable cases are treated in Subsections \ref{subsection4.3} and \ref{subsection4.4}, respectively.

\smallskip
\subsection{Regularity and embedding tools}\label{subsection4.1}

\paragraph{Multigraph regularity.}
In the proof of Theorem \ref{acycle}, we apply a version of Szemer\'{e}di's regularity lemma \cite{szemeredi}. Before we state it we need some more definitions. The \emph{density} of a bipartite graph $G=(A, B)$ with vertex classes $A$ and $B$ is defined to be
$$d_G (A,B):=\frac{e_G(A,B)}{|A||B|}.$$
We write $d(A,B)$ if this is unambiguous.
For any $\varepsilon >0$, we call $G$ {\it $\varepsilon$-regular} if for all $X\subseteq A$ and $Y \subseteq B$ with $|X|>\varepsilon |A|$ and $|Y|> \varepsilon |B|$, we have $|d(X,Y)-d(A,B)|<\varepsilon$. Note that when $G$ is a digraph, we similarly define the density and $\varepsilon$-regularity.

\smallskip

Given disjoint vertex set $A$ and $B$ in a graph $G$, we write $(A, B)$ for the induced bipartite subgraph of $G$ whose vertex classes are $A$ and $B$. If $G$ is a multigraph then we write $(A,B)^i _G$ for the bipartite graph where $a \in A$ and $b \in B$ are adjacent in $(A,B)^i _{G}$ precisely if $\mu(ab)=i$ in $G$.
Here $\mu (ab)$ is the multiplicity of the edge $ab$, i.e., the number of parallel edges joining $a$ and $b$ in $G$. An edge $ab$ with $\mu(ab)=1$ is called \emph{light} and an edge with $ab$ with $\mu(ab)=2$ is called \emph{heavy}. This symbol is used with a parallel meaning in the following text.
Finally, we introduce two further pieces of terminology. A loopless multigraph $G$ is said to be \emph{standard} if $\mu_G(xy)\leq 2$ for every pair of distinct vertices $x,y\in V(G)$. Moreover, a standard multigraph $G$ is said to be \emph{complete} if $\mu_G(xy)=2$ for every pair of distinct vertices $x,y\in V(G)$.

\smallskip

We apply the following version of the regularity lemma, which is an immediate corollary of a 2-coloured regularity lemma from \cite{blssw} (Theorem 2.4). This result in turn is easy to derive from the many-colour regularity lemma presented in \cite{Komlos11} (Theorem $1.18$).
\begin{lemma}[Degree form of multigraph regularity lemma]\label{Regularity Lemma}
For any $\varepsilon>0$ and  $M^\prime\in \mathbb{N}$, there exists $M=M(\varepsilon, M^\prime)$ such that the following holds. Let $G$ be a standard multigraph $(\mu(xy)\leq2$ for all $x, y\in V(G))$ on $n$ vertices and let $0\leq d\leq 1$. Then there exists a partition $\{V_0, V_1,\dots, V_k\}$ of $V(G)$ with $M^\prime \leq k \leq M$ and a spanning subgraph $G^\prime$ of $G$ with the following properties:\\
$(i)$ $|V_0| \leq \varepsilon n$;\\
$(ii)$ all vertex sets $V_i$, $i\in [k]$, are of the same size $\frac{(1-\varepsilon)n}{M}\leq \frac{n-|V_0|}{k}= |V_1|\leq \frac{n}{M^\prime}$;\\
$(iii)$ $d_{G^\prime}(v)>d_G(v)-(4d+2\varepsilon)n$ for all $v\in V(G^\prime)$;\\
$(iv)$ $e(G^\prime[V_i]) = 0$ for all $i \in [k]$;\\
$(v)$ for all $1 \leq i < j \leq k$ and $c\in [2]$, the pair $(V_i,V_j)^c _{G^\prime}$ is $\varepsilon$-regular with density either $0$ or at least $d$.
\end{lemma}
We refer to $V_1,\dots,V_k$ as the \emph{clusters}, to $V_0$ as the \emph{exceptional set}, and to $G^\prime$ as the \emph{pure multigraph}. Let $G$ be a multigraph and let $\varepsilon,d>0$ and $M^\prime\in\mathbb{N}$.
Apply Lemma~\ref{Regularity Lemma} with parameters
$\varepsilon$, $d$, and $M^\prime$ to obtain a partition $V(G)=V_0\cup V_1\cup\cdots\cup V_k$ and a pure multigraph $G^\prime$. The corresponding \emph{reduced multigraph}, denoted by $\Gamma$, has vertex set $V(\Gamma)=\{V_1,\dots,V_k\}.$ For each pair of distinct clusters $V_i$ and $V_j$, with $1\le i<j\le k$,
we define their adjacency in $\Gamma$ as follows:
\begin{enumerate}
    \item[(1)] If the density of $(V_i,V_j)^2_{G^\prime}$ is at least $d$,
    then $V_iV_j$ is a heavy edge of $\Gamma$.

    \item[(2)] If $(V_i,V_j)^2_{G^\prime}$ has density zero and
    $(V_i,V_j)^1_{G^\prime}$ has density at least $d$, then
    $V_iV_j$ is a light edge of $\Gamma$.

    \item[(3)] In all other cases, $V_i$ and $V_j$ are non-adjacent in $\Gamma$.
\end{enumerate}

The next result from \cite{Czygrinow} implies that the minimum degree of a multigraph is almost inherited by its reduced multigraph.
\begin{lemma}[{\cite{Czygrinow}}]\label{reduced multigraph degree}
Let $\varepsilon>0$, $d\in [0,1]$, $M^\prime, n \in \mathbb{N}$ and let $G$ be a multigraph on $n$ vertices. Let $G^\prime$ be the pure multigraph and $\Gamma$ be the reduced multigraph obtained by applying Lemma \ref{Regularity Lemma} to $G$ with parameters $\varepsilon$, $d$ and $M^\prime$. Then $\delta(\Gamma)\geq (\delta(G)/n-(8d+6\varepsilon))|\Gamma|$.
\end{lemma}
To complete the proof of Theorem \ref{acycle}, we shall also need an additional result. Let $\overline{\mathcal K}_r$ denote the family of standard multigraphs obtained from the complete standard multigraph on $r$ vertices by removing one copy of each edge in a matching. Equivalently, in every member of $\overline{\mathcal K}_r$ the light pairs form a matching and every other pair is heavy.
\begin{lemma}[{\cite{Czygrinow}}] \label{lem:noextremal tiling}
Let $n, r\in\mathbb{N}$ where $r\geq 2$ and $\eta>0$ such that $0<1/n\ll\eta\ll1/r$. Suppose that $G$ is a standard multigraph on $n$ vertices such that $\delta(G) \geq 2(1-1/r+\eta)n$.
Then $G$ contains an $\overline{\mathcal K}_r$-tiling covering all but at most $\eta n$ vertices.
\end{lemma}
\paragraph{Digraph regularity and embedding.}
Similarly, for a digraph $D$ and disjoint vertex sets $A, B\subseteq V(D)$, the \emph{density} of $(A, B)$ is defined by $d_D(A, B) :=\frac{e_D(A, B)}{|A||B|}.$ Note that $d_D(A, B)$ is not necessarily equal to $d_D(B, A)$. Given $\varepsilon>0$, we say that $(A, B)$ is \emph{$\varepsilon$-regular} in $D$ if for all subsets $A^\prime\subseteq A$ and $B^\prime\subseteq B$ with $|A^\prime|>\varepsilon|A|$ and $|B^\prime|>\varepsilon|B|$ we have that $$|d_D(A, B)-d_D(A^\prime, B^\prime)|<\varepsilon.$$

\smallskip

We now state the \emph{degree form} of the diregularity lemma, which can be directly derived from the standard version, and so we omit its proof.
\begin{lemma} [Degree form of the Diregularity lemma] \label{regular}
Given any $\varepsilon\in(0, 1)$ and $t_0\in\mathbb{N}$, there exist $T=T(\varepsilon, t_0)\in\mathbb{N}$ and $n_0=n_0(\varepsilon, t_0)\in\mathbb{N}$ such that the following holds for all $n\geq n_0$. Let $D$ be an $n$-vertex digraph and let $d\in[0, 1]$. There exists a partition $\{V_0, V_1, \ldots, V_t\}$ of $V(D)$ with $t_0\leq t \leq T$ and a spanning subdigraph $D^\prime$ of $D$ such that\\
$(i)$ $|V_0| \leq \varepsilon n$;\\
$(ii)$ $|V_i|=|V_1|$ for every $i\in[t]$;\\
$(iii)$ $e(D^\prime[V_i])=0$ for every $i\in[t]$;\\
$(iv)$ for every $x\in V(D)$, $d^+_{D^\prime}(x)> d^+_D(x)-(d+\varepsilon)n$ and $d^-_{D^\prime} (x)>d^-_D(x)-(d + \varepsilon)n$;\\
$(v)$ for every distinct $i, j\in[t]$, the pair $(V_i, V_j)$ is $\varepsilon$-regular in $D^\prime$ with density either $0$ or at least $d$.
\end{lemma}
\emph{The reduced digraph $R$ of $D$} with parameters $\varepsilon, d$ and $t_0$ is the digraph defined by
\begin{center}
$V(R):=\{V_1, \ldots, V_t\}$\ \mbox{and}\ $A(R):=\{V_iV_j: d_{D^\prime}(V_i, V_j)\geq d\}$.
\end{center}
We shall repeatedly use the following elementary consequence of regularity.
\begin{lemma}\label{lem:typical-vertices}
Let $(X,Y)$ be an $\varepsilon$-regular ordered pair in a digraph $D$ with density at least $d$. If $Y'\subseteq Y$ and $|Y'|>\varepsilon|Y|$, then all but at most $\varepsilon|X|$ vertices $x\in X$ satisfy
\[
d_D^+(x,Y')\geq(d-\varepsilon)|Y'|.
\]
The analogous assertion holds for in-neighborhoods.
\end{lemma}
\begin{proof}
Suppose that more than $\varepsilon|X|$ vertices of $X$ have fewer than $(d-\varepsilon)|Y'|$ out-neighbors in $Y'$. Choose a set $X'$ of more than $\varepsilon|X|$ such vertices. Then $d_D(X',Y')<d-\varepsilon\leq d_D(X,Y)-\varepsilon$, contradicting the $\varepsilon$-regularity of $(X,Y)$. The in-neighborhood version follows by applying the same argument to the reversed ordered pair.
\end{proof}

\smallskip

\noindent\emph{Remark.} If $(X,Y)$ is an $\varepsilon$-regular ordered pair of density at least $d$, and $X'\subseteq X$, $Y'\subseteq Y$ satisfy $|X'|\geq\alpha|X|$ and $|Y'|\geq\alpha|Y|$ for some $\alpha>\varepsilon$, then $(X',Y')$ has density at least $d-\varepsilon$ and is $\max\{2\varepsilon,\varepsilon/\alpha\}$-regular. Indeed, regularity gives $|d(X',Y')-d(X,Y)|<\varepsilon$. If $X''\subseteq X'$ and $Y''\subseteq Y'$ have sizes greater than $(\varepsilon/\alpha)|X'|$ and $(\varepsilon/\alpha)|Y'|$, respectively, then $|X''|>\varepsilon|X|$ and $|Y''|>\varepsilon|Y|$. Applying regularity to both $(X'',Y'')$ and $(X',Y')$ gives $|d(X'',Y'')-d(X',Y')|<2\varepsilon$. In the applications below $\alpha=1-o(1)$, so the inherited pairs are $2\varepsilon$-regular with density at least $d/2$ for sufficiently large $n$.

In particular, $C_4^2$ is the square of the directed cycle on cyclically ordered vertices $1,2,3,4$: the four cyclic pairs carry one forward arc, while each opposite pair carries both arcs.
\begin{lemma}\label{lemma4.9}
Let $0<1/m\ll\varepsilon\ll\eta\ll d\leq 1/2$. Suppose that $H$ is
a digraph, and let $U_1,U_2,U_3,U_4$ be pairwise disjoint sets of
size $m$. Suppose that every ordered pair corresponding to an arc of
$C_4^2$, namely
\[
(U_1,U_2),\ (U_2,U_3),\ (U_3,U_4),\ (U_4,U_1) \quad \mbox{and}\quad (U_1,U_3),\ (U_3,U_1),\ (U_2,U_4),\ (U_4,U_2),
\]
is $\varepsilon$-regular in $H$ and has density at least $d$. Let
$a_1a_2,b_1b_2\in A(H)$, where $a_1\in U_1$, $a_2\in U_2$,
$b_1\in U_3$, and $b_2\in U_4$, and suppose that
\[
|N_H^+(a_1)\cap N_H^+(a_2)\cap U_3|\geq\eta^2m,\qquad
|N_H^+(a_2)\cap U_4|\geq\eta m,
\]
\[
|N_H^-(b_1)\cap U_1|\geq\eta m,\qquad
|N_H^-(b_1)\cap N_H^-(b_2)\cap U_2|\geq\eta^2m.
\]
Then $H[U_1\cup U_2\cup U_3\cup U_4]$ contains a $2$-path from
$a_1a_2$ to $b_1b_2$ covering all but at most $4\eta m$ vertices.
\end{lemma}

\begin{proof}
Set $\kappa:=\eta^3/20$, $\lambda:=\eta^2/20$,
$\rho:=\eta^4/100$, and $\theta:=\eta/2$. By the hierarchy of
constants, we may assume that $3\varepsilon<\rho$,
$2\varepsilon<(d-\varepsilon)\rho$,
$(d-\varepsilon)\eta^2/2\geq\kappa$,
$(d-\varepsilon)\eta/2\geq\lambda$,
$(d-\varepsilon)\lambda\geq\kappa$, and
$(d-\varepsilon)\theta\geq\lambda$.

We shall repeatedly use the following immediate consequence of
Lemma~\ref{lem:typical-vertices}. Suppose that $Q$ is contained in
one cluster and $|Q|>r\varepsilon m$. If $S_1,\dots,S_r$ are subsets
of other clusters, each of size greater than $\varepsilon m$, and
the relevant ordered pairs are $\varepsilon$-regular of density at
least $d$, then one can choose a vertex of $Q$ that is typical with
respect to every $S_j$ in the required direction. Indeed, for each
$j\in[r]$, at most $\varepsilon m$ vertices of the cluster containing
$Q$ fail the corresponding degree bound.

\smallskip

We first show that, whenever $Z_1\subseteq U_3$, $Z_2\subseteq U_4$,
$X\subseteq U_1$, and $Y\subseteq U_2$ each have size at least
$\rho m$, there exist $z_1\in Z_1$, $z_2\in Z_2$, $x\in X$, and
$y\in Y$ such that $z_1z_2xy$ is a $2$-path.

Using the regular ordered pairs $(U_3,U_4)$, $(U_4,U_1)$, and
$(U_4,U_2)$, choose $z_2\in Z_2$ such that
$|N_H^-(z_2)\cap Z_1|\geq(d-\varepsilon)|Z_1|$,
$|N_H^+(z_2)\cap X|\geq(d-\varepsilon)|X|$, and
$|N_H^+(z_2)\cap Y|\geq(d-\varepsilon)|Y|$. Such a choice is
possible since $|Z_2|\geq\rho m>3\varepsilon m$.

Set $Z_1':=N_H^-(z_2)\cap Z_1$, $X':=N_H^+(z_2)\cap X$, and
$Y':=N_H^+(z_2)\cap Y$. Each of these sets has size at least
$(d-\varepsilon)\rho m>2\varepsilon m$. Using the regular ordered
pairs $(U_3,U_1)$ and $(U_1,U_2)$, choose $x\in X'$ such that
$|N_H^-(x)\cap Z_1'|\geq(d-\varepsilon)|Z_1'|$ and
$|N_H^+(x)\cap Y'|\geq(d-\varepsilon)|Y'|$. In particular,
$N_H^-(x)\cap Z_1'\neq\emptyset$ and
$N_H^+(x)\cap Y'\neq\emptyset$. We may therefore choose
$z_1\in N_H^-(x)\cap Z_1'$ and $y\in N_H^+(x)\cap Y'$.
Consequently, the arcs $z_1z_2$, $z_2x$, $xy$, $z_1x$, and $z_2y$
are present in $H$, and hence $z_1z_2xy$ is a $2$-path.

\smallskip

Choose sets $X\subseteq N_H^-(b_1)\cap U_1$ and
$Y\subseteq N_H^-(b_1)\cap N_H^-(b_2)\cap U_2$, each of size
$\rho m$ and avoiding the prescribed vertices. These choices are
possible by the assumed neighbourhood bounds and the hierarchy
$\rho\ll\eta^2$.

For each $i\in[4]$, let
$W_i:=U_i\setminus\bigl(X\cup Y\cup\{a_1,a_2,b_1,b_2\}\bigr)$.
Whenever a vertex is added to the path under construction, we remove
it from the corresponding set $W_i$. All subscripts below are taken
modulo $4$. An arc $xy$, where $x\in U_i$ and $y\in U_{i+1}$, is called
\emph{extendible} if
$|N_H^+(x)\cap N_H^+(y)\cap W_{i+2}|\geq\kappa m$ and
$|N_H^+(y)\cap W_{i+3}|\geq\lambda m$.

We first extend the prescribed initial arc $a_1a_2$. Let
$C_3:=N_H^+(a_1)\cap N_H^+(a_2)\cap W_3$ and
$C_4:=N_H^+(a_2)\cap W_4$. At this stage, only $b_1$ and $b_2$
have been excluded from $U_3$ and $U_4$, respectively. Hence, for
sufficiently large $m$, we have
$|C_3|\geq\eta^2m-1\geq\eta^2m/2$ and
$|C_4|\geq\eta m-1\geq\eta m/2$.

Using the regular ordered pairs $(U_3,U_4)$ and $(U_3,U_1)$, choose
$x_3\in C_3$ such that
$|N_H^+(x_3)\cap C_4|\geq(d-\varepsilon)|C_4|$ and
$|N_H^+(x_3)\cap W_1|\geq(d-\varepsilon)|W_1|$. This is possible
since $|C_3|>2\varepsilon m$. Add $x_3$ to the path and remove it
from $W_3$. Since $x_3\in C_3$, the sequence $a_1a_2x_3$ is a
$2$-path. Moreover,
$|N_H^+(a_2)\cap N_H^+(x_3)\cap W_4|
\geq(d-\varepsilon)|C_4|\geq\kappa m$. Also,
$|W_1|\geq(1-\rho)m-1\geq\theta m$, and hence
$|N_H^+(x_3)\cap W_1|\geq(d-\varepsilon)|W_1|\geq\lambda m$.
Thus the current terminal arc $a_2x_3$ is extendible.

We next show that an extendible terminal arc can be extended by one
vertex while preserving extendibility. Suppose that the path
currently ends with an extendible arc $xy$, where $x\in U_i$ and
$y\in U_{i+1}$, and suppose that $|W_j|\geq\theta m$ for every
$j\in[4]$. Set
$C:=N_H^+(x)\cap N_H^+(y)\cap W_{i+2}$ and
$R:=N_H^+(y)\cap W_{i+3}$. By extendibility,
$|C|\geq\kappa m>2\varepsilon m$ and
$|R|\geq\lambda m>\varepsilon m$.

Using the regular ordered pairs $(U_{i+2},U_{i+3})$ and
$(U_{i+2},U_i)$, choose $z\in C$ such that
$|N_H^+(z)\cap R|\geq(d-\varepsilon)|R|$ and
$|N_H^+(z)\cap W_i|\geq(d-\varepsilon)|W_i|$. Add $z$ to the path
and remove it from $W_{i+2}$. Since
$z\in N_H^+(x)\cap N_H^+(y)$, the sequence $xyz$ is a $2$-path.
Furthermore,
$|N_H^+(y)\cap N_H^+(z)\cap W_{i+3}|
\geq(d-\varepsilon)\lambda m\geq\kappa m$, while
$|N_H^+(z)\cap W_i|\geq(d-\varepsilon)\theta m\geq\lambda m$.
Hence the new terminal arc $yz$ is extendible.

Starting from $a_2x_3$, apply the preceding extension step three
times, adding vertices successively from $U_4$, $U_1$, and $U_2$.
Together with $x_3$, these extensions use exactly one vertex from
each set $W_i$. The path now ends with an extendible arc $pq$, where
$p\in U_1$ and $q\in U_2$.

We continue in blocks of four consecutive extensions. Each block
uses exactly one vertex from each $W_i$ and again leaves an
extendible terminal arc in $U_1\times U_2$. A new block is initiated
whenever $|W_i|\geq\theta m+1$ for every $i\in[4]$. Throughout such
a block, every available set has size at least $\theta m$, so all
four extension steps are valid.

Terminate the procedure when no further block can be initiated, and
let $pq\in U_1\times U_2$ be the final terminal arc. After every
completed block, each $W_i$ has size at least $\theta m$, whereas
the stopping condition gives an index $j\in[4]$ such that
$|W_j|<\theta m+1$.

Initially, the sizes of the four sets $W_i$ differ by at most
$\rho m+4$. Both the initial four-cluster extension and every
subsequent block use the same number of vertices from each cluster.
These differences therefore remain at most $\rho m+4$ throughout
the construction. Consequently,
$\theta m\leq|W_i|<\theta m+\rho m+5$ for every $i\in[4]$.

Since $pq$ is extendible, the sets
$Z_1:=N_H^+(p)\cap N_H^+(q)\cap W_3$ and
$Z_2:=N_H^+(q)\cap W_4$ satisfy $|Z_1|\geq\kappa m$ and
$|Z_2|\geq\lambda m$. In particular, both sets have size at least
$\rho m$. Applying the four-cluster connection to $Z_1,Z_2,X,Y$,
we obtain vertices $z_1\in Z_1$, $z_2\in Z_2$, $x\in X$, and
$y\in Y$ such that $z_1z_2xy$ is a $2$-path.

By the definitions of $Z_1$ and $Z_2$, the sequence $pqz_1z_2xy$
is a $2$-path. Moreover, the definitions of $X$ and $Y$ yield
$xb_1,yb_1,yb_2\in A(H)$. Together with $b_1b_2\in A(H)$, this
shows that $pqz_1z_2xyb_1b_2$ is a $2$-path. Concatenating this
segment with the path already constructed gives a $2$-path $P$
from $a_1a_2$ to $b_1b_2$.

For $i\in\{3,4\}$, the uncovered vertices of $U_i$ are contained
in $W_i$, and hence there are at most $\theta m+\rho m+5$ such
vertices. For $i\in\{1,2\}$, we must also account for the unused
vertices of the corresponding reserved set, giving at most
$\theta m+2\rho m+5$ uncovered vertices. Since
$\theta=\eta/2$, $\rho=\eta^4/100$, and $m$ is sufficiently large,
both quantities are smaller than $\eta m$. Thus $P$ leaves fewer
than $\eta m$ vertices uncovered in each cluster, and therefore at
most $4\eta m$ vertices uncovered in total.
\end{proof}
Finally, we record an orientation lemma.

\begin{lemma}\label{lemmq4.10}
Let $R$ be a digraph, and let $\Gamma$ be the standard multigraph
on $V(R)$ defined by
\[
\mu_\Gamma(XY):=
\mathbf 1_{\{XY\in A(R)\}}+
\mathbf 1_{\{YX\in A(R)\}}
\]
for all distinct $X,Y\in V(R)$. Let $K$ be a submultigraph of $\Gamma$ with $K\in\overline{\mathcal K}_4$, with vertex set
$\{W_1,W_2,W_3,W_4\}$. Then these vertices can be ordered as
$X_1,X_2,X_3,X_4$ so that $R$ contains every arc of $C_4^2$
in this order; that is, $X_iX_{i+1}\in A(R)$ for every $i\in[4]$, with indices taken modulo $4$, and $X_1X_3, X_3X_1, X_2X_4, X_4X_2\in A(R).$
\end{lemma}

\begin{proof}
The light edges of $K$ form a matching of size zero, one, or two,
and every other pair is heavy. If there is no light edge, any ordering
works. If there is one light edge, say $X\to Y$, put $X_1=X$,
$X_2=Y$, and let $X_3,X_4$ be the remaining vertices. If there
are two light edges, say $X\to Y$ and $Z\to W$, order the
vertices as $X,Y,Z,W$. In every case, the light edges occur as
cycle edges and all other required pairs are heavy.
\end{proof}
\subsection{Common setup for the extremal theorem}\label{subsection4.2}

Choose constants and an integer $t_0$ satisfying
\[
0<\frac1{t_0}\ll\varepsilon\ll d\ll\eta_0\ll\gamma,
\qquad
\gamma^{1/4}\ll\beta\ll\omega\ll\xi<\frac15.
\]
Let $T=T(\varepsilon,t_0)$ be the bound supplied by
Lemma~\ref{regular}, and then choose $n_0$ sufficiently large that
$1/n_0\ll1/T$. As usual, the hierarchy is taken sufficiently strong
to imply all numerical inequalities used below. Throughout the
remainder of this section, let $D$ be an $n$-vertex digraph satisfying
the assumptions of Theorem~\ref{acycle}, where $n\geq n_0$.

\medskip

\noindent\textbf{Step 1: Cleaning the extremal set.}

\smallskip

Since $D$ is $(1/5,\gamma)$-extremal, there exists a set
$S_0\subseteq V(D)$ such that
$\bigl||S_0|-n/5\bigr|<\gamma n$ and
$e(D[S_0])<\gamma n^2$. By discarding vertices if necessary, choose
a set $S\subseteq S_0$ with
$|S|=\lfloor(1/5-\gamma)n\rfloor$, and put
\[
D^*:=D-S
\quad\text{and}\quad
n^*:=|V(D^*)|=n-|S|=(4/5+\gamma)n+O(1).
\]

We first remove from $S$ the vertices having atypically large degree
inside $S$. Define
\[
S':=\{v\in S:d_S(v)\geq\gamma^{1/2}n\}
\quad\text{and}\quad
A:=S\setminus S'.
\]
Since $\sum_{v\in S}d_S(v)=2e(D[S])<2\gamma n^2$, we have
$|S'|\leq2\gamma^{1/2}n$. Hence, for sufficiently large $n$,
\begin{equation}\label{4.1}
|A|\geq(1/5-3\gamma^{1/2})n.
\end{equation}

The set $A$ has the additional advantage that each of its vertices
has almost full indegree and outdegree in $D^*$. Indeed, for every
$a\in A$, we have
$d_D(a,V(D^*))\geq\tau(n)-\gamma^{1/2}n$. Since each of
$d_D^+(a,V(D^*))$ and $d_D^-(a,V(D^*))$ is at most $n^*$, it follows
that each is at least
$\tau(n)-\gamma^{1/2}n-n^*$. Using
$\tau(n)\geq8n/5-2$ and
$n^*\leq(4/5+\gamma)n+1$, we obtain, for sufficiently large $n$,
\begin{equation}\label{4.2}
d_D^+(a,V(D^*)),d_D^-(a,V(D^*))
\geq n^*-2\gamma^{1/2}n
\quad\text{for every }a\in A.
\end{equation}
Thus $A$ is large, and every vertex of $A$ has almost full
indegree and outdegree in $D^*$.

\medskip

\noindent\textbf{Step 2: Establishing the required degree structure.}

\smallskip

Removing $S$ decreases the total degree of any vertex by at most
$2|S|$. Consequently,
$\delta(D^*)\geq\delta(D)-2|S|
\geq(6/5+2\gamma)n-2$. Since
$n^*=(4/5+\gamma)n+O(1)$, for sufficiently large $n$ this gives
\begin{equation}\label{4.3}
\delta(D^*)\geq(3/2+\gamma/2)n^*.
\end{equation}
As each one-sided degree in $D^*$ is at most $n^*$, we immediately
deduce that
\begin{equation}\label{4.4}
d_{D^*}^+(v),d_{D^*}^-(v)
\geq(1/2+\gamma/2)n^*
\quad\text{for every }v\in V(D^*).
\end{equation}

We next show that almost every vertex of $D^*$ has large degree in
both directions into $S$. Summing \eqref{4.2} over $A$ and using
$|S\setminus A|=|S'|\leq2\gamma^{1/2}n$, we obtain
\begin{equation}\label{4.5}
\begin{aligned}
e_D(V(D^*),S)
&\geq\sum_{a\in A}d_D^-(a,V(D^*))
 \geq n^*|S|-3\gamma^{1/2}n^2,\\
e_D(S,V(D^*))
&\geq\sum_{a\in A}d_D^+(a,V(D^*))
 \geq n^*|S|-3\gamma^{1/2}n^2.
\end{aligned}
\end{equation}
Furthermore, we define that $U^\prime_+:=\{u\in V(D^*):d_S^+(u)<(1-\gamma^{1/4})|S|\},$ $U^\prime_-:=\{u\in V(D^*):d_S^-(u)<(1-\gamma^{1/4})|S|\}$ and $U^\prime:=U^\prime_+\cup U^\prime_-.$ Clearly, each vertex of $U'_+$ is missing more than
$\gamma^{1/4}|S|$ possible arcs from that vertex to $S$, whereas
the first inequality in \eqref{4.5} shows that the total number of
such missing arcs is at most $3\gamma^{1/2}n^2$. Since
$|S|\geq n/6$ for sufficiently large $n$, it follows that
$|U'_+|\leq18\gamma^{1/4}n$. Applying the same argument to the
second inequality in \eqref{4.5} gives
$|U'_-|\leq18\gamma^{1/4}n$, and hence
$|U'|\leq36\gamma^{1/4}n$.

Moreover, every $v\in V(D^*)\setminus U'$ satisfies
$d_A^+(v)\geq d_S^+(v)-|S'|
\geq |A|-\gamma^{1/4}|S|$. The same argument gives the corresponding
lower bound for $d_A^-(v)$. Thus, after removing the small exceptional
set $U'$, every remaining vertex of $D^*$ has almost full indegree
and outdegree into $A$.

\medskip

\noindent\textbf{Step 3: Developing the insertion and connection tools.}

\smallskip

We now establish two auxiliary tools that will be used in the
subsequent construction of the desired $2$-cycle. The first allows
several vertices to be inserted into an existing $2$-cycle without
interference, while the second provides a short $2$-path between
two prescribed arcs in the stable case.

Let $C_0$ be a $2$-cycle or a $2$-path, and let
$y_1,y_2,y_3,y_4$ be four consecutive vertices of $C_0$. For a
vertex $x\notin V(C_0)$, we say that $x$ can be \emph{inserted} at
the arc $y_2y_3$ if $y_1x,y_2x,xy_3,xy_4\in A(D)$. Replacing
$y_2y_3$ by $y_2xy_3$ then preserves the $2$-path or $2$-cycle
structure. Let $\operatorname{Pos}(x)$ denote the set of arcs of
$C_0$ at which $x$ can be inserted.

\begin{claim}\label{claim4.11}
Let $C_0$ be a $2$-cycle in $D$, and let
$x_1,\ldots,x_q\notin V(C_0)$ be distinct. If
$|\operatorname{Pos}(x_j)|\geq3(j-1)+1$ for every $j\in[q]$, then
all the vertices $x_1,\ldots,x_q$ can be inserted into $C_0$,
producing a $2$-cycle on
$V(C_0)\cup\{x_1,\ldots,x_q\}$.
\end{claim}

\begin{proof}
\renewcommand*{\qedsymbol}{$\blacksquare$}
All insertion positions are considered with respect to the original
cycle $C_0$. Insert $x_1,\ldots,x_q$ successively in this order.
Whenever a vertex is inserted at an original arc $e$, declare $e$
and the two original arcs adjacent to $e$ unavailable for all
subsequent insertions.

An insertion at $e$ changes the four-vertex configurations
associated only with these three original arcs. Consequently, every
original arc not declared unavailable retains the same four
consecutive original vertices in the required cyclic order. Before
$x_j$ is inserted, at most $3(j-1)$ original arcs have therefore
been declared unavailable.

Since $|\operatorname{Pos}(x_j)|\geq3(j-1)+1$, there is an
available arc $e_j\in\operatorname{Pos}(x_j)$. The four original
vertices certifying that $e_j\in\operatorname{Pos}(x_j)$ remain
consecutive in the required order, so the same four arcs still
permit the insertion of $x_j$. Proceeding inductively proves the
claim.
\end{proof}

For later reference, we collect the estimates obtained so far.
\begin{equation}\label{4.7}
\begin{gathered}
|S'|\leq2\gamma^{1/2}n,\quad
|A|\geq(1/5-3\gamma^{1/2})n,\quad
|U'|\leq36\gamma^{1/4}n,\quad
\delta^0(D^*)\geq(1/2+\gamma/2)n^*,\\
d_D^+(a,V(D^*)),d_D^-(a,V(D^*))
\geq n^*-2\gamma^{1/2}n
\quad\text{for every }a\in A,\\
d_A^+(v),d_A^-(v)
\geq |A|-\gamma^{1/4}|S|
\quad\text{for every }v\in V(D^*)\setminus U'.
\end{gathered}
\end{equation}

The next lemma uses these estimates to join two prescribed arcs by
a short $2$-path whose internal vertices consist of two vertices of
$A$ and an arc of $D^*$.
\begin{figure}[H]
\centering

\begin{minipage}[c]{0.82\linewidth}
\centering

\resizebox{\linewidth}{!}{%
\begin{tikzpicture}[
x=1cm,
y=1cm,
line cap=round,
line join=round
]

\node[aset] (A) at (0,3.08) {};

\node[font=\normalsize,anchor=south]
at ($(A.north)+(0,0.06)$) {$A$};

\node[bridgepoint] (ai) at (-0.37,2.93) {};
\node[bridgepoint] (bi) at ( 0.37,2.93) {};

\node[bridgepoint] (zi) at (-0.48,1.48) {};
\node[bridgepoint] (wi) at ( 0.48,1.48) {};

\begin{scope}[xshift=0.90cm]

\foreach \i/\x in {
1/-6.05,
2/-4.70,
3/-3.35,
4/-2.00
}{
\node[cluster] (L\i) at (\x,0) {};
}

\draw[blackarrowline] (L1) -- (L2);
\draw[blackline] (L2) -- (L3);
\draw[blackarrowline] (L3) -- (L4);

\draw[blackline]
(L1.-58)
.. controls (-5.35,-0.86) and (-3.95,-0.90) ..
(L3.-112);

\draw[blackline]
(L1.-65)
.. controls (-5.20,-1.40) and (-3.10,-1.42) ..
(L4.-100);

\draw[blackline]
(L2.-70)
.. controls (-4.28,-0.88) and (-3.18,-0.92) ..
(L4.-108);

\end{scope}

\begin{scope}[xshift=-0.90cm]

\foreach \i/\x in {
1/2.00,
2/3.35,
3/4.70,
4/6.05
}{
\node[cluster] (R\i) at (\x,0) {};
}

\draw[blackarrowline] (R1) -- (R2);
\draw[blackline] (R2) -- (R3);
\draw[blackarrowline] (R3) -- (R4);

\draw[blackline]
(R1.-58)
.. controls (2.70,-0.90) and (4.10,-0.86) ..
(R3.-112);

\draw[blackline]
(R1.-65)
.. controls (3.10,-1.42) and (5.20,-1.40) ..
(R4.-100);

\draw[blackline]
(R2.-70)
.. controls (4.18,-0.92) and (5.28,-0.88) ..
(R4.-108);

\end{scope}

\node[font=\small,anchor=south]
at ($(L1.north)+(0.05,0.02)$) {$V_1^i$};

\node[font=\small,anchor=south]
at ($(L2.north)+(0.03,0.02)$) {$V_2^i$};

\node[font=\small,anchor=south]
at ($(L3.north)+(0.035,0.02)$) {$V_3^i$};

\node[font=\small,anchor=south]
at ($(L4.north)+(-0.10,0.02)$) {$V_4^i$};

\node[font=\small,anchor=south]
at ($(R1.north)+(0.18,0.02)$) {$V_1^{i+1}$};

\node[font=\small,anchor=south]
at ($(R2.north)+(0.10,0.02)$) {$V_2^{i+1}$};

\node[font=\small,anchor=south]
at ($(R3.north)+(0.05,0.02)$) {$V_3^{i+1}$};

\node[font=\small,anchor=south]
at ($(R4.north)+(0.07,0.02)$) {$V_4^{i+1}$};

\coordinate (pL3) at (L3.center);
\coordinate (pL4) at (L4.center);
\coordinate (pR1) at (R1.center);
\coordinate (pR2) at (R2.center);

\draw[redline] (pL3) -- (ai);
\draw[redline] (pL4) -- (ai);
\draw[redline] (pL4) -- (zi);

\draw[redline] (ai) -- (zi);
\draw[redline] (ai) -- (wi);

\draw[
blue,
line width=1.15pt,
shorten >=0pt,
shorten <=0pt,
postaction={decorate},
decoration={
markings,
mark=at position 0.55 with {\arrow{stealth}}
}
] (zi) -- (wi);

\draw[redline] (zi) -- (bi);
\draw[redline] (wi) -- (bi);
\draw[redline] (wi) -- (pR1);

\draw[redline] (bi) -- (pR1);
\draw[redline] (bi) -- (pR2);

\node[contactdot] at (pL3) {};
\node[contactdot] at (pL4) {};
\node[contactdot] at (pR1) {};
\node[contactdot] at (pR2) {};

\node[font=\scriptsize]
at ($(L3.center)+(0,-0.23)$) {$x_3^i$};

\node[font=\scriptsize]
at ($(L4.center)+(0,-0.23)$) {$x_4^i$};

\node[font=\small,anchor=south east]
at ($(ai)+(0,-0.10)$) {$a_i$};

\node[font=\small,anchor=south west]
at ($(bi)+(0,-0.10)$) {$b_i$};

\node[font=\small,anchor=east]
at ($(zi)+(0.46,-0.28)$) {$z_i$};

\node[font=\small,anchor=west]
at ($(wi)+(-0.46,-0.28)$) {$w_i$};

\node[font=\scriptsize]
at ($(R1.center)+(0.05,-0.23)$) {$x_1^{i+1}$};

\node[font=\scriptsize]
at ($(R2.center)+(0.05,-0.23)$) {$x_2^{i+1}$};

\node[font=\normalsize] at (-3.10,-1.68) {$K^i$};
\node[font=\normalsize] at ( 3.10,-1.68) {$K^{i+1}$};

\end{tikzpicture}%
}

\par\smallskip
{\small\textnormal{(a)} The $\beta$-stable case.}

\end{minipage}

\par\vspace{1.1em}


\begin{minipage}[c]{0.5\linewidth}
\centering

\resizebox{\linewidth}{!}{%
\begin{tikzpicture}[
x=1cm,
y=1cm,
line cap=round,
line join=round
]

\node[setcircle] (L) at (-2.75,1.30) {};
\node[setcircle] (R) at ( 2.75,1.30) {};
\node[setcircleA] (A) at ( 0.00,3.85) {};

\node[
font=\large,
anchor=south
]
at ($(L.north)+(0,0.12)$) {$L_0$};

\node[
font=\large,
anchor=south
]
at ($(R.north)+(0,0.12)$) {$R_0$};

\node[
font=\large,
anchor=south
]
at ($(A.north)+(0,0.12)$) {$A$};

\manualstealtharc{L}{0.37}{240}{490}{515}{-17}
\manualstealtharc{R}{0.37}{240}{490}{515}{-17}

\draw[mainarrow]
(R.west) -- (L.east);

\draw[mainarrow]
($(L.north east)+(0.10,-0.11)$) --
($(A.south west)+(0.10,-0.11)$);

\draw[mainarrow]
($(A.south west)+(-0.10,0.11)$) --
($(L.north east)+(-0.10,0.11)$);

\draw[mainarrow]
($(A.south east)+(-0.10,-0.11)$) --
($(R.north west)+(-0.10,-0.11)$);

\draw[mainarrow]
($(R.north west)+(0.10,0.11)$) --
($(A.south east)+(0.10,0.11)$);

\end{tikzpicture}%
}

\par\smallskip
{\small\textnormal{(b)} The non-$\beta$-stable case.}

\end{minipage}

\caption{Schematic structures in the two cases of the extremal
argument. In \textnormal{(a)}, the highlighted arcs form the
connecting segment joining the $2$-paths associated with two
consecutive tiles $K^i$ and $K^{i+1}$. In \textnormal{(b)}, the
subdigraphs $D[L_0]$ and $D[R_0]$ are $10\omega$-almost complete,
the ordered pair $(R_0,L_0)$ is
$2\omega^{1/4}$-almost one-way complete, while $(L_0,R_0)$ is
$\beta^2$-sparse. Moreover, the pairs between $A$ and each of
$L_0$ and $R_0$ are $\gamma^{1/4}$-almost complete in both
directions.}
\label{fig:extremal-two-cases}
\end{figure}
\begin{lemma}\label{lemma12}
Suppose that $D^*$ is $\beta$-stable. Let $pq$ and $rs$ be two
disjoint arcs of $D^*-U'$, and let
$F\subseteq\bigl(V(D^*)\cup A\bigr)\setminus\{p,q,r,s\}$ satisfy
$|F|\leq8T$. Then there exist distinct vertices
$a,b\in A\setminus F$ and
$z,w\in V(D^*)\setminus
\bigl(U'\cup F\cup\{p,q,r,s\}\bigr)$ such that $pqazwbrs$ is a
$2$-path.
\end{lemma}

\begin{proof}
Since $p,q\notin U'$, the inequality \eqref{4.7} gives
$|N_A^+(p)\cap N_A^+(q)\setminus F|
\geq |A|-2\gamma^{1/4}|S|-8T>0$. Choose
$a\in N_A^+(p)\cap N_A^+(q)\setminus F$.

For this choice of $a$, define
\[
P:=N_{D^*}^+(q)\cap N_D^+(a,V(D^*))
\setminus\bigl(U'\cup F\cup\{p,q,r,s\}\bigr),\ \mbox{and}
\]
\[
Q:=N_{D^*}^-(r)\cap N_D^+(a,V(D^*))
\setminus\bigl(U'\cup F\cup\{p,q,r,s\}\bigr).
\]
Every vertex $z\in P$ receives arcs from both $q$ and $a$, whereas
every vertex $w\in Q$ receives an arc from $a$ and sends an arc to
$r$. It therefore remains to find an arc directed from $P$ to $Q$.

By \eqref{4.7}, we get that
\[
\begin{aligned}
|P|
\geq d_{D^*}^+(q)+d_D^+(a,V(D^*))-n^*
      -|U'|-|F|-4
\geq
\left(\frac12+\frac\gamma2\right)n^*
-2\gamma^{1/2}n-36\gamma^{1/4}n-8T-4.
\end{aligned}
\]
The same estimate holds for $|Q|$, with $d_{D^*}^-(r)$ in place of
$d_{D^*}^+(q)$. Since $n/n^*\leq5/4$, the parameter hierarchy and
the choice of $n_0$ imply that
$|P|,|Q|\geq(1/2-\beta)n^*$. By the $\beta$-stability of $D^*$, we
therefore have $e_{D^*}(P,Q)>(\beta n^*)^2$. In particular, there
exist $z\in P$ and $w\in Q$ such that $zw\in A(D^*)$.

The vertices $z,w,r,s$ all lie outside $U'$. Applying the last
estimate in \eqref{4.7} to these four vertices gives
\[
\begin{aligned}
|N_A^+(z)\cap N_A^+(w)\cap
  N_A^-(r)\cap N_A^-(s)\setminus(F\cup\{a\})|
\geq
|A|-4\gamma^{1/4}|S|-|F|-1>0.
\end{aligned}
\]
Choose $b$ from this intersection. By construction, we have that $\{pq,qa,az,zw,wb,br,rs,pa,qz,aw,zb,\\ wr,bs\}\subseteq A(D)$. Therefore
$pqazwbrs$ is a $2$-path, as required.
\end{proof}

We next treat the two cases separately.
\subsection{The $\beta$-stable case}\label{subsection4.3}

\begin{lemma}\label{lem:stable-extremal}
Under the setup of Subsection~\ref{subsection4.2}, if $D^*$ is
$\beta$-stable, then $D$ contains a $2$-cycle covering all but at
most $\xi n$ vertices.
\end{lemma}

\begin{proof}
We begin by applying Lemma~\ref{regular} to $D^*$ with parameters
$\varepsilon,d,t_0$. This yields a partition
\[
V(D^*)=V_0\cup V_1\cup\cdots\cup V_t,
\qquad t_0\le t\le T,
\]
together with a pure spanning subdigraph $D'$ and the associated
reduced digraph $R$. Let $m:=|V_1|=\cdots=|V_t|$, and define a
standard multigraph $\Gamma$ on $\{V_1,\ldots,V_t\}$ by
\[
\mu_\Gamma(V_iV_j):=
\mathbf 1_{\{V_iV_j\in A(R)\}}+
\mathbf 1_{\{V_jV_i\in A(R)\}}.
\]
Thus the multiplicity of $V_iV_j$ records the number of orientations
between $V_i$ and $V_j$ that occur in $R$.

By the degree conclusion of Lemma~\ref{regular}, every
$v\in V(D^*)$ satisfies
$d_{D'}(v)>d_{D^*}(v)-2(d+\varepsilon)n^*$. Fix $i\in[t]$ and
$v\in V_i$. Since $D'[V_i]$ is empty and the contribution of $V_0$
to $d_{D'}(v)$ is at most $2|V_0|$, we have
$d_{D'}(v)\leq m d_\Gamma(V_i)+2|V_0|$. It follows that
\begin{equation*}
\delta(\Gamma)
\ge\left(\frac{\delta(D^*)}{n^*}-2d-4\varepsilon\right)t
\ge(3/2+\gamma/3)t.
\end{equation*}
By construction, $\Gamma$ is loopless and every pair of vertices
has multiplicity at most two, so $\Gamma$ is standard. Since
$\eta_0\ll\gamma$, the preceding estimate gives
$\delta(\Gamma)\geq2(1-1/4+\eta_0)t$. Moreover,
$1/t\leq1/t_0\ll\eta_0\ll1/4$. Lemma~\ref{lem:noextremal tiling},
applied with $r=4$ and parameter $\eta_0$, therefore yields an
$\overline{\mathcal K}_4$-tiling of $\Gamma$ covering all but at
most $\eta_0t$ clusters.

\smallskip

We next discard those tiles containing too many vertices of the
exceptional set $U'$. Call a cluster $V_j$ to be \emph{$U'$-heavy} if
$|V_j\cap U'|\geq m/2$, and let $b$ denote the number of
$U'$-heavy clusters. Since $bm/2\leq|U'|$, the estimates in
\eqref{4.7}, together with $m=(n^*-|V_0|)/t$, imply that
$b\leq200\gamma^{1/4}t$. Discard all clusters not covered by the
tiling, together with every tile containing a $U'$-heavy cluster,
and denote the remaining tiles by $K^1,\ldots,K^r$. The total number
of discarded vertices, including those in $V_0$, is at most
\begin{equation}\label{4.9}
(2\varepsilon+\eta_0+800\gamma^{1/4})n^*.
\end{equation}
In particular, every cluster belonging to a retained tile contains
more than $m/2$ vertices outside $U'$.

\smallskip

Fix $i\in[r]$. By Lemma~\ref{lemmq4.10}, the clusters of $K^i$ may
be labelled $V_1^i,V_2^i,V_3^i,V_4^i$ so that $R$ contains all arcs
of $C_4^2$ in this cyclic order. Set
$W_j^i:=V_j^i\setminus U'$; then $|W_j^i|>m/2$ for every
$j\in[4]$. We now select an initial arc in
$W_1^i\times W_2^i$ and a terminal arc in
$W_3^i\times W_4^i$ with the neighbourhood properties required for
Lemma~\ref{lemma4.9}.

By regularity, there exists $x_1^i\in W_1^i$ such that
$|N_{D'}^+(x_1^i)\cap W_2^i|\geq(d-\varepsilon)|W_2^i|$ and
$|N_{D'}^+(x_1^i)\cap V_3^i|\geq(d-\varepsilon)m$. Put
\[
X_2^i:=N_{D'}^+(x_1^i)\cap W_2^i,\qquad
X_3^i:=N_{D'}^+(x_1^i)\cap V_3^i.
\]
A second application of regularity gives a vertex
$x_2^i\in X_2^i$ such that
$|N_{D'}^+(x_2^i)\cap X_3^i|
\geq(d-\varepsilon)|X_3^i|$ and
$|N_{D'}^+(x_2^i)\cap V_4^i|\geq(d-\varepsilon)m$. Consequently,
$x_1^ix_2^i\in A(D')$,
$|N_{D'}^+(x_1^i)\cap N_{D'}^+(x_2^i)\cap V_3^i|
\geq(d-\varepsilon)^2m$, and
$|N_{D'}^+(x_2^i)\cap V_4^i|\geq(d-\varepsilon)m$.

The terminal arc is chosen analogously. Select
$x_4^i\in W_4^i$ such that
$|N_{D'}^-(x_4^i)\cap W_3^i|
\geq(d-\varepsilon)|W_3^i|$ and
$|N_{D'}^-(x_4^i)\cap V_2^i|\geq(d-\varepsilon)m$, and put
\[
Y_3^i:=N_{D'}^-(x_4^i)\cap W_3^i,\qquad
Y_2^i:=N_{D'}^-(x_4^i)\cap V_2^i.
\]
Choose $x_3^i\in Y_3^i$ such that
$|N_{D'}^-(x_3^i)\cap V_1^i|\geq(d-\varepsilon)m$ and
$|N_{D'}^-(x_3^i)\cap Y_2^i|
\geq(d-\varepsilon)|Y_2^i|$. We then have
$x_3^ix_4^i\in A(D')$,
$|N_{D'}^-(x_3^i)\cap V_1^i|\geq(d-\varepsilon)m$, and
$|N_{D'}^-(x_3^i)\cap N_{D'}^-(x_4^i)\cap V_2^i|
\geq(d-\varepsilon)^2m$.

\smallskip

We now join the terminal arc of each tile to the initial arc of the
next tile, with indices taken modulo $r$. At the $i$th application
of Lemma~\ref{lemma12}, let the forbidden set consist of all selected
end-vertices other than
$x_3^i,x_4^i,x_1^{i+1},x_2^{i+1}$, together with the internal
vertices of all previously constructed connecting paths. Since
$r\leq T$, this forbidden set has size at most $8T$.
Lemma~\ref{lemma12} therefore yields vertices
$a_i,b_i\in A$ and $z_i,w_i\in V(D^*)\setminus U'$ such that
$x_3^ix_4^ia_iz_iw_ib_ix_1^{i+1}x_2^{i+1}$ is a $2$-path. By the
choice of the forbidden sets, the resulting connecting paths are
pairwise internally vertex-disjoint, and their internal vertices
are disjoint from all selected end-vertices. Set
$B^{(i)}:=a_iz_iw_ib_i$ and regard this sequence as the connecting
segment from the $i$th tile to the $(i+1)$st tile.

\smallskip

The sets $W_j^i$ were introduced only to ensure that the four
selected end-vertices of each tile lie outside $U'$; the vertices
of $U'$ are not removed from the clusters used in the subsequent
embedding. We do, however, remove every vertex among
$\{z_i,w_i:i\in[r]\}$ from its cluster, since these vertices have
already been assigned to the connecting segments. We then trim the
retained clusters to a common size $m'$, without deleting any
selected end-vertex. Since there are at most $2T$ vertices of the
form $z_i$ or $w_i$ altogether, the common size may be chosen so
that $m'\geq m-2T\geq n^*/(2T)$. The total number of additional
vertices removed in this equalisation is $O(T^2)=o(n^*)$.

\smallskip

For each fixed $i$, the eight ordered pairs
\[
(V_1^i,V_2^i),(V_2^i,V_3^i),(V_3^i,V_4^i),(V_4^i,V_1^i)\quad \mbox{and}\quad (V_1^i,V_3^i),(V_3^i,V_1^i),(V_2^i,V_4^i),(V_4^i,V_2^i)
\]
were $\varepsilon$-regular with density at least $d$ before the
trimming. Since each cluster loses only $O(T)=o(m)$ vertices, the
slicing lemma implies that the resulting ordered pairs are
$2\varepsilon$-regular with density at least $d/2$. The choice of
the selected end-vertices, together with the deletion of at most
$2T=o(m)$ vertices from each cluster, also gives
\[
|N_{D'}^+(x_1^i)\cap N_{D'}^+(x_2^i)\cap V_3^i|
\geq\frac{d^2}{4}m',\qquad
|N_{D'}^+(x_2^i)\cap V_4^i|\geq\frac d3m',
\]
\[
|N_{D'}^-(x_3^i)\cap V_1^i|\geq\frac d3m',\qquad
|N_{D'}^-(x_3^i)\cap N_{D'}^-(x_4^i)\cap V_2^i|
\geq\frac{d^2}{4}m'.
\]

We may now apply Lemma~\ref{lemma4.9} with
$\varepsilon':=2\varepsilon$, $d':=d/2$, and
$\eta':=d^3/1000$. The hierarchy of constants gives
$1/m'\ll\varepsilon'\ll\eta'\ll d'$, and the preceding paragraph
verifies the eight regular-pair assumptions. Moreover,
$d^2/4\geq(\eta')^2$ and $d/3\geq\eta'$, so the selected initial
and terminal arcs satisfy the required endpoint conditions. Hence
each tile contains a $2$-path
$P_i=x_1^ix_2^i\cdots x_3^ix_4^i$ covering all but at most
$4\eta'm'=d^3m'/250\leq dm'$ vertices of its four clusters.

\smallskip

The paths $P_1,\ldots,P_r$, together with the connecting segments
$B^{(1)},\ldots,B^{(r)}$, concatenate to form the $2$-cycle
\[
C_0=P_1B^{(1)}P_2B^{(2)}\cdots P_rB^{(r)}.
\]
By \eqref{4.9}, the equalisation of the cluster sizes, and
Lemma~\ref{lemma4.9},
\begin{equation}\label{4.10}
\begin{gathered}
n^*-|V(C_0)\cap V(D^*)|
\leq(3\varepsilon+\eta_0+800\gamma^{1/4}+d)n^*+o(n^*)\ \mbox{and}\
|V(C_0)\cap S|=2r\leq2T.
\end{gathered}
\end{equation}

It remains to incorporate the unused vertices of $A$. Let
$A_0:=A\setminus V(C_0)$ and fix $a\in A_0$. We estimate the number
of insertion positions for $a$ with respect to the original cycle
$C_0$. A missing arc $xa$, where
$x\in V(C_0)\cap V(D^*)$, can affect only the two insertion tests in
which $x$ is one of the two predecessors of the inserted vertex.
Similarly, a missing arc $ax$ can affect only the two tests in which
$x$ is one of the two successors. By \eqref{4.2}, the vertex $a$
has at most $2\gamma^{1/2}n$ missing out-arcs and at most
$2\gamma^{1/2}n$ missing in-arcs to $D^*$. The missing adjacencies
between $a$ and $D^*$ therefore eliminate at most
$8\gamma^{1/2}n$ insertion positions.

We have no corresponding degree information between $a$ and the
vertices of $V(C_0)\cap S$. Nevertheless, each such vertex occurs
in at most four insertion tests, twice as a predecessor and twice
as a successor. Since \eqref{4.10} gives
$|V(C_0)\cap S|=2r\leq2T$, these vertices eliminate at most $8T$
additional positions. Consequently,
$|\operatorname{Pos}(a)|\geq|C_0|-8\gamma^{1/2}n-8T$.

\smallskip

By \eqref{4.10}, we have
$|C_0|\geq n^*-
(3\varepsilon+\eta_0+800\gamma^{1/4}+d)n^*-o(n^*)$. The parameter
choices ensure that
$(3\varepsilon+\eta_0+d)n^*<n/40$ and
$(800\gamma^{1/4}+8\gamma^{1/2})n<n/20$. After increasing $n_0$ if
necessary, we may also assume that $o(n^*)<n/40$ and $8T<n/20$.
Since $n^*\geq4n/5$, these estimates imply
$|\operatorname{Pos}(a)|>3n/5$. On the other hand,
$|A_0|\leq|S|<n/5$, and hence
$|\operatorname{Pos}(a)|>3|A_0|$ for every $a\in A_0$.

Write $A_0=\{a_1,\ldots,a_q\}$. Since the relevant cardinalities
are integers, the preceding strict inequality gives
$|\operatorname{Pos}(a_j)|\geq3q+1\geq3(j-1)+1$ for every
$j\in[q]$. Claim~\ref{claim4.11} therefore allows all vertices of
$A_0$ to be inserted into $C_0$. The resulting $2$-cycle $C$
satisfies
$$n-|C| \leq |S'|+\bigl(n^*-|V(C_0)\cap V(D^*)|\bigr)
\leq2\gamma^{1/2}n+
(3\varepsilon+\eta_0+800\gamma^{1/4}+d)n^*+o(n)
\leq\xi n.$$
This completes the proof of Lemma \ref{lem:stable-extremal}.
\end{proof}
\subsection{The non-$\beta$-stable case}\label{subsection4.4}

\begin{lemma}\label{lem:nonstable-extremal}
Under the setup of Subsection~\ref{subsection4.2}, if $D^*$ is not
$\beta$-stable, then $D$ contains a $2$-cycle covering all but at
most $\xi n$ vertices.
\end{lemma}

\begin{proof}
The proof proceeds in three stages. We first obtain two nearly
balanced parts $L_0$ and $R_0$ which are internally dense and for
which almost all cross-arcs are directed from $R_0$ to $L_0$. We
then show that there is nevertheless a bounded $2$-path from $L_0$
to $R_0$; this is the principal difficulty, as the path runs against
the predominant orientation of the cross-arcs. Finally, we combine
this path with a short connection in the reverse direction, find
almost-spanning $2$-paths inside $L_0$ and $R_0$, and insert the
unused vertices of $A$ into the resulting $2$-cycle.

We retain the notation $S,S',A,D^*$ and $U'$ introduced above, as
well as the estimates \eqref{4.1}--\eqref{4.7} and
Claim~\ref{claim4.11}. Recall in particular that
$A=S\setminus S'$, $D^*=D-S$, and $U'$ is the exceptional set
defined immediately before \eqref{4.7}. Thus
$V(D)=V(D^*)\cup A\cup S'$. We shall also use throughout that
$\gamma^{1/4}\ll\beta\ll\omega\ll\xi$, that $300\omega<1/5$, and
that $\delta(D)\geq\tau(n)$. In particular,
$\delta^0(D)\geq\delta(D)-(n-1)\geq\tau(n)-(n-1)\geq3n/5-1$.

We begin by deriving the approximate bipartite structure forced by
the failure of $\beta$-stability.

\smallskip

\begin{claim}
There is a partition $V(D^*)=L\cup R\cup Z$ such that
\begin{equation}\label{4.13}
|Z|\leq4\beta n^*,\quad
|L|,|R|=\left(\frac12\pm2\beta\right)n^*\quad \mbox{and}
\quad e_{D^*}(L,R)\leq\beta^2(n^*)^2.
\end{equation}
\end{claim}

\begin{proof}
\renewcommand*{\qedsymbol}{$\blacksquare$}
Since $D^*$ is not $\beta$-stable, there exist sets
$X,Y\subseteq V(D^*)$ such that
$|X|,|Y|\geq(1/2-\beta)n^*$ and
$e_{D^*}(X,Y)\leq\beta^2(n^*)^2$. Define
\[
M:=X\cap Y,\qquad L:=X\setminus Y,\qquad
R:=Y\setminus X,\qquad Q:=V(D^*)\setminus(X\cup Y).
\]

For every $v\in M$, \eqref{4.3} gives
$d_{D^*}^+(v,Y)+d_{D^*}^-(v,X)
\geq d_{D^*}(v)-(n^*-|Y|)-(n^*-|X|)
\geq(1/2+\gamma/2-2\beta)n^*
\geq(1/2-3\beta)n^*$.
Summing over $v\in M$, each arc directed from $X$ to $Y$ is counted
at most twice. Hence
$|M|(1/2-3\beta)n^*\leq2e_{D^*}(X,Y)
\leq2\beta^2(n^*)^2$, and therefore $|M|\leq6\beta^2n^*$.

Furthermore,
$|Q|=n^*-|X|-|Y|+|M|\leq2\beta n^*+6\beta^2n^*
\leq3\beta n^*$. Setting $Z:=M\cup Q$, we obtain
$|Z|\leq|M|+|Q|\leq4\beta n^*$. Moreover,
\[
|L|=|X|-|M|
\geq\left(\frac12-\beta-6\beta^2\right)n^*
\geq\left(\frac12-2\beta\right)n^*,
\]
and the same lower bound holds for $|R|$. Since $L,R,Z$ partition
$V(D^*)$, these lower bounds also imply
$|L|,|R|\leq(1/2+2\beta)n^*$. Finally,
$L\subseteq X$ and $R\subseteq Y$, and hence
$e_{D^*}(L,R)\leq e_{D^*}(X,Y)\leq\beta^2(n^*)^2$, proving
\eqref{4.13}.
\end{proof}

We next remove the few vertices that do not conform to the
predominant orientation from $R$ to $L$. Define
$L_{\rm exc}:=\{x\in L:d_{D^*}^+(x,R)>\omega n^*\}$ and
$R_{\rm exc}:=\{y\in R:d_{D^*}^-(y,L)>\omega n^*\}$. By
\eqref{4.13},
\begin{equation}\label{4.15}
|L_{\rm exc}|,\ |R_{\rm exc}|
\leq\frac{\beta^2}{\omega}n^*.
\end{equation}
Set $L_0:=L\setminus(L_{\rm exc}\cup U')$ and
$R_0:=R\setminus(R_{\rm exc}\cup U')$. Using \eqref{4.15} and
$|U'|\leq36\gamma^{1/4}n$, we obtain
\[
|L\setminus L_0|,\ |R\setminus R_0|
\leq\frac{\beta^2}{\omega}n^*+36\gamma^{1/4}n
\leq\omega n^*.
\]
Together with \eqref{4.13}, the hierarchy
$\gamma^{1/4}\ll\beta\ll\omega$, and
$n^*=(4/5+\gamma)n+O(1)$, this yields
\begin{equation}\label{4.16}
|L_0|,|R_0|
=\left(\frac12\pm3\omega\right)n^*
=\left(\frac25\pm4\omega\right)n.
\end{equation}

Moreover,
$V(D^*)\setminus(L_0\cup R_0)
\subseteq Z\cup L_{\rm exc}\cup R_{\rm exc}\cup U'$. Hence
\eqref{4.13}, \eqref{4.15}, and the bound on $|U'|$ give $|V(D^*)\setminus(L_0\cup R_0)|
\leq4\beta n^*+\frac{2\beta^2}{\omega}n^*
     +36\gamma^{1/4}n
\leq8\omega n.$ Since $V(D)=V(D^*)\cup A\cup S'$ and
$|S'|\leq2\gamma^{1/2}n$, we obtain
\begin{equation}\label{4.17}
|V(D^*)\setminus(L_0\cup R_0)|\leq8\omega n,
\qquad
|V(D)\setminus(A\cup L_0\cup R_0)|\leq10\omega n.
\end{equation}

Thus, apart from at most $10\omega n$ vertices, the vertex set of
$D$ is partitioned into $A,L_0$ and $R_0$. The following claim
records the degree properties of these three parts that will be used
throughout the remainder of the proof.

\smallskip

\begin{claim}
For each $v\in L_0\cup R_0$, let $B_v$ be the member of
$\{L_0,R_0\}$ containing $v$, let $\overline B_v$ be the other
member, and set $\sigma_v:=-$ if $v\in L_0$ and $\sigma_v:=+$ if
$v\in R_0$. Then
\begin{equation}\label{4.18}
d_{B_v}^{\pm}(v)\geq|B_v|-10\omega n,
\qquad
d_{\overline B_v}^{\sigma_v}(v)
\geq|\overline B_v|-10\omega n.
\end{equation}
Moreover, we get that
\begin{align}
&d_A^\pm(v)\geq|A|-\gamma^{1/4}|S|
\quad\text{for every }v\in L_0\cup R_0,\ \mbox{and}\label{4.22} \\
&d_D^\pm(a,L_0)\geq|L_0|-3\gamma^{1/2}n
\quad \mbox{and}\quad
d_D^\pm(a,R_0)\geq|R_0|-3\gamma^{1/2}n
\quad\text{for every }a\in A. \label{4.23}
\end{align}
\end{claim}

\begin{proof}
\renewcommand*{\qedsymbol}{$\blacksquare$}
We prove \eqref{4.18} for a vertex $x\in L_0$. The corresponding
assertions for vertices of $R_0$ follow by symmetry, after
interchanging $L$ and $R$ and reversing the orientation of every
arc.

Since $x\notin L_{\rm exc}$, we have
$d_{D^*}^+(x,R)\leq\omega n^*$. Hence \eqref{4.4} gives
\[
d_{D^*}^+(x,L)
\geq d_{D^*}^+(x)-d_{D^*}^+(x,R)-|Z|
\geq\left(\frac12+\frac{\gamma}{2}\right)n^*
     -\omega n^*-4\beta n^*
\geq|L|-2\omega n^*.
\]
For the last inequality, we used
$|L|\leq(1/2+2\beta)n^*$ and $\beta\ll\omega$.

We next estimate the indegrees of $x$ using \eqref{4.3}. Since
$d_{D^*}^+(x,L)\leq|L|-1$,
$d_{D^*}^+(x,R)\leq\omega n^*$,
$d_{D^*}^-(x,R)\leq|R|$, and $d_{D^*}(x,Z)\leq2|Z|$, we have that $d_{D^*}(x)
\leq d_{D^*}^-(x,L)+(|L|-1)+\omega n^*+|R|+2|Z|.$ It follows that
$d_{D^*}^-(x,L)\geq
\delta(D^*)-(|L|-1)-|R|-\omega n^*-2|Z|
\geq|L|-3\omega n^*$. Indeed, after subtracting
$|L|-3\omega n^*$, the resulting lower bound is at least
$(\gamma/2+2\omega-6\beta)n^*>0$, where we used
$|L|+|Z|\leq(1/2+6\beta)n^*$.

\smallskip

Similarly,
$d_{D^*}(x)\leq d_{D^*}^-(x,R)+2|L|+\omega n^*+2|Z|$, and hence
$d_{D^*}^-(x,R)\geq|R|-3\omega n^*$. This follows from the same
calculation, since
$\delta(D^*)-n^*-|L|-|Z|+2\omega n^*
\geq(\gamma/2+2\omega-6\beta)n^*>0$.

\smallskip

Finally, since
$|L\setminus L_0|,|R\setminus R_0|\leq\omega n^*$, restricting the
preceding estimates to $L_0$ and $R_0$ gives
$d_{L_0}^+(x)\geq|L_0|-2\omega n^*$,
$d_{L_0}^-(x)\geq|L_0|-3\omega n^*$, and
$d_{R_0}^-(x)\geq|R_0|-3\omega n^*$. Since $n^*\leq n$, these
bounds imply \eqref{4.18} for $x\in L_0$, and the symmetric argument
proves the assertion for vertices of $R_0$.

\smallskip

Estimate \eqref{4.22} follows directly from \eqref{4.7}, since
$L_0\cup R_0\subseteq V(D^*)\setminus U'$. Finally, \eqref{4.2}
shows that every $a\in A$ has at most $2\gamma^{1/2}n$ missing
out-neighbours and at most $2\gamma^{1/2}n$ missing in-neighbours in
$V(D^*)$. Thus
$d_D^\pm(a,L_0)\geq|L_0|-2\gamma^{1/2}n$ and
$d_D^\pm(a,R_0)\geq|R_0|-2\gamma^{1/2}n$, which are stronger than
\eqref{4.23}.
\end{proof}

We next record a simple consequence of an almost-complete minimum
semidegree condition.

\smallskip

\begin{claim}\label{claim4.15}
Let $G$ be a digraph of order $m$, let $k\in\mathbb N$, and suppose
that $\delta^0(G)\geq m-k$ and $m\geq10k+6$. Then, for any two
disjoint arcs $a_1a_2,b_1b_2\in A(G)$, there is a $2$-path from
$a_1a_2$ to $b_1b_2$ leaving at most $5k$ vertices of $G$
uncovered.
\end{claim}

\begin{proof}
\renewcommand*{\qedsymbol}{$\blacksquare$}
Begin with the $2$-path $a_1a_2$, keeping $b_1$ and $b_2$ outside
the path until the final step. At any stage, let $W$ be the set of
vertices not yet used on the path, excluding $b_1$ and $b_2$.
Initially, $|W|=m-4\geq10k+2>5k+2$. Suppose that the current
terminal arc is $xy$. Whenever $|W|>5k+2$, the bound
$|N_G^+(x)\cap N_G^+(y)\cap W|\geq|W|-2k>0$ allows us to extend the
path by one vertex of $W$.

Continue until $|W|=5k+2$. There are at least
$|W|-3k\geq2k+2$ choices for a vertex
$u\in W\cap N_G^+(x)\cap N_G^+(y)\cap N_G^-(b_1)$. Having chosen
$u$, there are at least $|W|-1-4k\geq k+1$ choices for a vertex
$v\in(W\setminus\{u\})\cap N_G^+(y)\cap N_G^+(u)
\cap N_G^-(b_1)\cap N_G^-(b_2)$. Extending the path by
$u,v,b_1,b_2$ completes the construction, and precisely
$|W|-2=5k$ vertices remain uncovered.
\end{proof}

We now show that the predominant orientation from $R_0$ to $L_0$
cannot prevent a short connection in the opposite direction.

\smallskip

\begin{claim}\label{claim4.16}
There is a $2$-path $B^+$ of order at most $9$ whose initial arc
lies in $D[L_0]$ and whose terminal arc lies in $D[R_0]$.
\end{claim}

\begin{proof}
\renewcommand*{\qedsymbol}{$\blacksquare$}
Suppose, to the contrary, that no such $2$-path exists. The hierarchy
of constants and \eqref{4.16} ensure that
\[
|A|-4\gamma^{1/4}|S|
\geq\left(\frac15-3\gamma^{1/2}-\frac45\gamma^{1/4}\right)n>0
\]
and that both
$|L_0|-10\omega n-3\gamma^{1/2}n-1$ and
$|R_0|-10\omega n-3\gamma^{1/2}n-1$ are positive for sufficiently
large $n$.

We first derive two consequences of this assumption. The first is that
$e_D(L_0,R_0)=0$. Indeed, suppose that $qr\in A(D)$ for some $q\in L_0$ and
$r\in R_0$. By \eqref{4.18}, we may choose
$p\in N_D^-(q,L_0)$ and $s\in N_D^+(r,R_0)$. Since
$p,q,r,s\notin U'$, \eqref{4.22} gives $|N_A^+(p)\cap N_A^+(q)\cap N_A^-(r)\cap N_A^-(s)|
\geq|A|-4\gamma^{1/4}|S|>0.$ Choosing $a$ in this intersection produces the forbidden $2$-path
$pqars$, a contradiction.

\smallskip

The second consequence is that $D[A]$ is independent. Suppose instead that $ab\in A(D[A])$. By \eqref{4.23}, the set
$N_D^-(a,L_0)\cap N_D^-(b,L_0)$ is nonempty; choose $q$ from this
intersection. The estimates \eqref{4.18} and \eqref{4.23} then give $|N_D^-(q,L_0)\cap N_D^-(a,L_0)\setminus\{q\}|
\geq|L_0|-10\omega n-3\gamma^{1/2}n-1>0,$ so we may choose $p$ in this set. Similarly, choose
$r\in N_D^+(a,R_0)\cap N_D^+(b,R_0)$ and then, using
\eqref{4.18} and \eqref{4.23}, choose $s\in N_D^+(r,R_0)\cap N_D^+(b,R_0)\setminus\{r\}.$ By construction, $pqabrs$ is a $2$-path from $L_0$ to $R_0$, again
a contradiction.

We now extend the three-part structure as far as possible. Consider
triples $(\widehat A,\widehat L,\widehat R)$ of pairwise disjoint
vertex sets satisfying $A\subseteq\widehat A$,
$L_0\subseteq\widehat L$, $R_0\subseteq\widehat R$, and
\begin{equation}\label{4.30}
D[\widehat A]\text{ is independent},
\qquad e_D(\widehat L,\widehat R)=0.
\end{equation}
Such a triple exists by the preceding two conclusions. Choose one
maximising $|\widehat A|+|\widehat L|+|\widehat R|$. Write
$\widehat a:=|\widehat A|$, $\widehat\ell:=|\widehat L|$,
$\widehat r:=|\widehat R|$, and let
$U:=V(D)\setminus(\widehat A\cup\widehat L\cup\widehat R)$ and
$u:=|U|$. The lower bounds in \eqref{4.1} and \eqref{4.16}, the
disjointness of the three sets, and \eqref{4.17} yield
$\left|\widehat a-\frac n5\right|,
\left|\widehat\ell-\frac{2n}{5}\right|,
\left|\widehat r-\frac{2n}{5}\right|,
u\leq20\omega n$.

Set $\Delta:=50\omega n+3$. We next show that every vertex of the
three maximal parts is adjacent to almost every vertex to which an
arc is permitted by \eqref{4.30}. If $a\in\widehat A$, then
\begin{align}\label{4.32}
2(n-\widehat a)-d_D(a)
\leq2(n-\widehat a)-\delta(D)
\leq2\left(n-\left(\frac15-20\omega\right)n\right)-\tau(n)
\leq40\omega n+2<\Delta.
\end{align}
For $\ell\in\widehat L$, the only forbidden incident arcs are those
directed from $\ell$ to $\widehat R$, so its maximum possible total
degree subject to \eqref{4.30} is
$2(\widehat\ell-1)+\widehat r+2\widehat a+2u
=2n-\widehat r-2$. Hence
$2n-\widehat r-2-d_D(\ell)
\leq2n-\widehat r-2-\delta(D)
\leq2n-(2/5-20\omega)n-2-\tau(n)
\leq20\omega n<\Delta$. Similarly, every $r\in\widehat R$ satisfies
$2n-\widehat\ell-2-d_D(r)
\leq2n-\widehat\ell-2-\delta(D)
\leq2n-(2/5-20\omega)n-2-\tau(n)
\leq20\omega n<\Delta$.
Thus each vertex of
$\widehat A\cup\widehat L\cup\widehat R$ is missing at most
$\Delta$ incident arcs permitted by \eqref{4.30}. We shall also use
that \eqref{4.17} implies
$|\widehat L\setminus L_0|,|\widehat R\setminus R_0|
\leq10\omega n<\Delta$, while \eqref{4.16} and
$300\omega<1/5$ give $|L_0|,|R_0|>2\Delta+9$ for sufficiently large
$n$.

The following selection principle will be used repeatedly. Suppose
that $W$ is contained in one of the three maximal parts, that $F$
is a set of previously selected vertices, and that a vertex
$w\in W\setminus F$ is required to satisfy $j$ prescribed
adjacency conditions with fixed vertices of
$\widehat A\cup\widehat L\cup\widehat R$. If all the required arcs
are permitted by \eqref{4.30}, then at most $j\Delta+|F\cap W|$
vertices of $W$ are unavailable. Thus a suitable choice exists
whenever
\begin{align}\label{4.34a}
|W|>j\Delta+|F\cap W|.
\end{align}
In every application below, $j\leq2$ and fewer than nine vertices
have already been chosen. The stronger condition
\begin{align}\label{4.34b}
|W|>2\Delta+9
\end{align}
is therefore sufficient throughout.

Suppose that $U\neq\emptyset$, and fix $x\in U$. Define
$I_L:=N_D^-(x,\widehat L)$,
$O_R:=N_D^+(x,\widehat R)$,
$I_A:=N_D^-(x,\widehat A)$, and
$O_A:=N_D^+(x,\widehat A)$. The maximality of the triple gives
$I_L\neq\emptyset$ and $O_R\neq\emptyset$, since otherwise $x$
could be added to $\widehat R$ or $\widehat L$, respectively. It also gives $I_A\cup O_A\neq\emptyset$, for otherwise $x$ could
be added to $\widehat A$.

All in-neighbours of $x$ outside $I_L\cup I_A$ lie in
$\widehat R\cup(U\setminus\{x\})$. Using the minimum semidegree of $D$ and recalling that
$\left|\widehat r-2n/5\right|,u\leq20\omega n$, we obtain
\begin{equation}\label{4.39}
|I_L|+|I_A|
\geq\tau(n)-(n-1)-\widehat r-u
\geq\left(\frac15-40\omega\right)n-1>4\Delta+4.
\end{equation}
The analogous outdegree estimate gives
\begin{equation}\label{4.40}
|O_A|+|O_R|
\geq\left(\frac15-40\omega\right)n-1>4\Delta+4.
\end{equation}
We now obtain a contradiction in each of the following three
subcases.

\medskip

\begin{case}\label{subcase2.1}
$I_A\neq\emptyset$ and $|O_R|>2\Delta+2$.
\end{case}

We first choose $q\in I_L$ and $a\in I_A$ such that $qa\in A(D)$.
If $|I_L|>\Delta$, fix any $a\in I_A$ and use \eqref{4.32}; at
most $\Delta$ vertices of $I_L$ fail to send an arc to $a$. If
$|I_L|\leq\Delta$, then \eqref{4.39} gives
$|I_A|>3\Delta+4$. Fixing any $q\in I_L$, the preceding deficit bound for vertices of
$\widehat L$ shows that at most $\Delta$ vertices of $I_A$ fail to
receive an arc from $q$.
Thus the required pair $q,a$ exists in either case.

Since $|\widehat R\setminus R_0|<\Delta$, we have
$|O_R\cap R_0|>\Delta+2$. Using the selection principle, choose
$r\in(O_R\cap R_0)\cap N_D^+(a)$ and then
$t\in(O_R\cap R_0\setminus\{r\})\cap N_D^+(r)$. Next choose,
avoiding all previously selected vertices,
\[
v\in N_D^-(q,L_0)\cap N_D^-(a,L_0),\quad
u\in N_D^-(v,L_0)\cap N_D^-(q,L_0),\quad
s\in N_D^+(r,R_0)\cap N_D^+(t,R_0).
\]
Each of these choices follows from \eqref{4.34a}--\eqref{4.34b}.
The resulting sequence $uvqaxrts$ is a $2$-path from $L_0$ to
$R_0$ of order eight, a contradiction.

\medskip

\begin{case}\label{subcase2.2}
$O_A\neq\emptyset$ and $|I_L|>2\Delta+2$.
\end{case}

We first choose $b\in O_A$ and $r\in O_R$ such that $br\in A(D)$.
If $|O_A|>\Delta$, fix any $r\in O_R$ and apply the preceding
deficit bound for vertices of $\widehat R$; if $|O_A|\leq\Delta$, then \eqref{4.40} gives
$|O_R|>3\Delta+4$, and \eqref{4.32} applies after fixing any
$b\in O_A$. Thus such a pair $b,r$ always exists.

Since $|\widehat L\setminus L_0|<\Delta$, we have
$|I_L\cap L_0|>\Delta+2$. Choose
$q\in(I_L\cap L_0)\cap N_D^-(b)$ and then
$p\in(I_L\cap L_0\setminus\{q\})\cap N_D^-(q)$. Finally choose,
avoiding all previously selected vertices,
\[
t\in N_D^+(r,R_0)\cap N_D^+(b,R_0),\qquad
s\in N_D^+(r,R_0)\cap N_D^+(t,R_0).
\]
Again \eqref{4.34a}--\eqref{4.34b} guarantee all choices. The
sequence $pqxbrts$ is a $2$-path from $L_0$ to $R_0$ of order
seven, a contradiction.

\medskip

\begin{case}
Neither Subcase~\ref{subcase2.1} nor
Subcase~\ref{subcase2.2} holds.
\end{case}

We first note that $I_A$ and $O_A$ are both nonempty. Indeed, if
$I_A=\emptyset$, then the maximality of the triple gives
$O_A\neq\emptyset$, while
\eqref{4.39} gives $|I_L|>4\Delta+4$, placing us in
Subcase~\ref{subcase2.2}. The proof that $O_A\neq\emptyset$ is
symmetric. Since neither of the first two subcases applies,
\eqref{4.39} and \eqref{4.40} now yield
$|I_A|>2\Delta+2$ and $|O_A|>2\Delta+2$.
Choose arbitrary vertices $q\in I_L$ and $r\in O_R$. By these
bounds and the deficit estimates, we may choose
$a\in I_A\cap N_D^+(q)\cap N_D^-(r)$ and then
$b\in(O_A\setminus\{a\})\cap N_D^+(r)$. Using the selection
principle, choose successively
\begin{align}
&v\in N_D^-(q,L_0)\cap N_D^-(a,L_0),\quad
u\in N_D^-(v,L_0)\cap N_D^-(q,L_0),\nonumber\\
&t\in N_D^+(r,R_0)\cap N_D^+(b,R_0)\quad \mbox{and}\quad
s\in N_D^+(b,R_0)\cap N_D^+(t,R_0),\nonumber
\end{align}
avoiding all previously selected vertices. The sequence
$uvqaxrbts$ is a $2$-path from $L_0$ to $R_0$ of order nine, again
a contradiction.

It follows that $U=\emptyset$. Thus
$V(D)=\widehat A\cup\widehat L\cup\widehat R$, with
$D[\widehat A]$ independent and $e_D(\widehat L,\widehat R)=0$.
For every $\ell\in\widehat L$, we have
$d_D(\ell)\leq2(|\widehat L|-1)+|\widehat R|+2|\widehat A|
=2n-|\widehat R|-2$. Hence
$|\widehat R|\leq2n-\tau(n)-2$, and symmetrically
$|\widehat L|\leq2n-\tau(n)-2$. Consequently, we have that $|\widehat A|=n-|\widehat L|-|\widehat R|
\geq2\tau(n)-3n+4.$ On the other hand, the independence of $D[\widehat A]$ gives
$d_D(a)\leq2(n-|\widehat A|)$ for every $a\in\widehat A$, and
therefore $|\widehat A|\leq\lfloor n-\tau(n)/2\rfloor$.
Writing $n=5q+r$, where $0\leq r\leq4$, the two bounds become
\[
\begin{array}{c|ccccc}
r&0&1&2&3&4\\ \hline
\tau(n)&8q-1&8q+1&8q+2&8q+4&8q+5\\
2\tau(n)-3n+4&q+2&q+3&q+2&q+3&q+2\\
\lfloor n-\tau(n)/2\rfloor&q&q&q+1&q+1&q+1
\end{array}
\]
In every residue class, the lower bound exceeds the upper bound,
a final contradiction. This proves Claim~\ref{claim4.16}.
\end{proof}

Let $B^+$ be a $2$-path supplied by Claim~\ref{claim4.16}, with
initial arc $\ell_1\ell_2\in A(D[L_0])$ and terminal arc
$r_1r_2\in A(D[R_0])$. We next construct a disjoint $2$-path in the
reverse direction. Choose $r_3\in R_0\setminus V(B^+)$ and, using
\eqref{4.18}, choose
$r_4\in N_D^+(r_3,R_0)\setminus V(B^+)$. Again by \eqref{4.18},
$r_3$ and $r_4$ have at least $|L_0|-20\omega n$ common
out-neighbours in $L_0$; choose one of them as $\ell_3$, avoiding
$V(B^+)$. Finally, the same estimate gives a vertex
$\ell_4\in N_D^+(r_4,L_0)\cap N_D^+(\ell_3,L_0)$ outside all
previously selected vertices. Since $B^+$ has bounded order, all
these choices are possible for sufficiently large $n$. Set $B^-:=r_3r_4\ell_3\ell_4$.

\smallskip

Remove from $L_0$ all vertices of $B^+\cup B^-$ other than
$\ell_1,\ell_2,\ell_3,\ell_4$, and denote the resulting set by
$L^\circ$. Define $R^\circ$ analogously, retaining
$r_1,r_2,r_3,r_4$. Since only the internal vertices of the two
bounded connecting paths are removed, \eqref{4.18} gives
$\delta^0(D[L^\circ])\geq|L^\circ|-10\omega n$ and
$\delta^0(D[R^\circ])\geq|R^\circ|-10\omega n$. In particular,
both semidegrees are at least the relevant order minus
$20\omega n$.

\smallskip

Set $k:=\lceil20\omega n\rceil$. By \eqref{4.16} and the bounded
orders of $B^+$ and $B^-$, for sufficiently large $n$ we have
$|L^\circ|,|R^\circ|\geq(2/5-5\omega)n$. Moreover,
$10k+6\leq200\omega n+16$, and the assumption $300\omega<1/5$
ensures that $(2/5-5\omega)n\geq10k+6$ for sufficiently large $n$.
Claim~\ref{claim4.15} therefore yields a $2$-path $P_L$ in
$D[L^\circ]$ from $\ell_3\ell_4$ to $\ell_1\ell_2$ leaving at most
$5k\leq100\omega n+5$ vertices of $L^\circ$ uncovered. Similarly,
there is a $2$-path $P_R$ in $D[R^\circ]$ from $r_1r_2$ to
$r_3r_4$ leaving at most $100\omega n+5$ vertices of $R^\circ$
uncovered.

\smallskip

Identifying each terminal arc with the corresponding initial arc,
the four paths concatenate to form the $2$-cycle
$C_0:=P_LB^+P_RB^-$. Every vertex of
$(L_0\cup R_0)\setminus(L^\circ\cup R^\circ)$ is an internal vertex
of one of the two connecting paths and therefore lies on $C_0$.
Consequently,
$V(D^*)\setminus V(C_0)$ is contained in the union of
$V(D^*)\setminus(L_0\cup R_0)$,
$L^\circ\setminus V(P_L)$, and $R^\circ\setminus V(P_R)$. By
\eqref{4.17} and the two applications of Claim~\ref{claim4.15}, for
sufficiently large $n$ we have
$|V(D^*)\setminus V(C_0)|
=n^*-|V(C_0)\cap V(D^*)|
\leq8\omega n+2(100\omega n+5)
\leq220\omega n$.
Only $B^+$ can contain vertices of $S$, and $|V(B^+)|\leq9$; hence
$|V(C_0)\cap S|\leq9$.

\smallskip

It remains to insert the vertices of
$A_0:=A\setminus V(C_0)$. Fix $a\in A_0$ and define
$\operatorname{Pos}(a)$ with respect to the original cycle $C_0$,
as in Claim~\ref{claim4.11}. By \eqref{4.2}, the vertex $a$ has at
most $2\gamma^{1/2}n$ missing out-neighbours and at most
$2\gamma^{1/2}n$ missing in-neighbours in $V(D^*)$. Each missing
directed adjacency eliminates at most two insertion positions, so
the missing adjacencies between $a$ and $V(D^*)$ eliminate at most
$8\gamma^{1/2}n$ positions. Each vertex of $V(C_0)\cap S$ can
eliminate at most four further positions. Therefore
$|\operatorname{Pos}(a)|
\geq|C_0|-8\gamma^{1/2}n-36$.
Recall that $n^*-|V(C_0)\cap V(D^*)|\leq220\omega n$. By the
definition of $n^*$, it follows that
$|C_0|\geq n^*-220\omega n
\geq(4/5+\gamma-220\omega)n$. On the other hand,
$|A_0|\leq|S|\leq(1/5-\gamma)n$. Combining these estimates gives
\[
|\operatorname{Pos}(a)|
\geq\left(\frac45+\gamma-220\omega-8\gamma^{1/2}\right)n-36
>3|A_0|,
\]
where the final inequality follows from $300\omega<1/5$,
$\gamma^{1/2}\ll\omega$, and the choice of $n_0$.

Write $A_0=\{a_1,\ldots,a_q\}$. By integrality,
$|\operatorname{Pos}(a_j)|\geq3q+1\geq3(j-1)+1$ for every
$j\in[q]$. Claim~\ref{claim4.11} therefore allows us to insert all
vertices of $A_0$ into $C_0$. Let $C$ be the resulting $2$-cycle.
Since every vertex of $A$ now lies on $C$, we have
$V(D)\setminus V(C)\subseteq S'\cup(V(D^*)\setminus V(C_0))$.
Consequently,
\[
n-|C|
\leq|S'|+\bigl(n^*-|V(C_0)\cap V(D^*)|\bigr)
\leq2\gamma^{1/2}n+220\omega n
\leq\xi n.
\]
This completes the proof of Lemma \ref{lem:nonstable-extremal}.
\end{proof}

\begin{proof}[\textbf{Proof of Theorem~\ref{acycle}}]
The common setup of Subsection~\ref{subsection4.2} applies. If
$D^*$ is $\beta$-stable, the conclusion follows from
Lemma~\ref{lem:stable-extremal}; otherwise it follows from
Lemma~\ref{lem:nonstable-extremal}.
\end{proof}

\medskip
\section{Completion of the proof of Theorem \ref{szc2}}\label{Section5}
We first isolate the completion argument common to the extremal and non-extremal cases.

\begin{lemma}\label{completion-lemma}
Let $D$ be an $n$-vertex digraph with $\delta(D)\geq\tau(n)$. If $D$ contains a $2$-cycle leaving fewer than $n/5$ vertices uncovered, then $D$ contains the square of a Hamilton cycle.
\end{lemma}

\begin{proof}
Choose a $2$-cycle $C=c_1\cdots c_mc_1$ of maximum order in $D$, with indices read modulo $m$, and put $r:=n-m$. By hypothesis, $r<n/5$.

We first observe that if $v\notin V(C)$ and $d_D(v,V(C))>3m/2$, then $v$ can be inserted into $C$. Indeed, suppose that no such insertion is possible. For every $i\in[m]$, inserting $v$ between $c_{i+1}$ and $c_{i+2}$ would require the four arcs $c_iv,c_{i+1}v,vc_{i+2},vc_{i+3}$. Hence at least one of these arcs is absent, and therefore
\[
\bigl|\{c_i,c_{i+1}\}\cap N_D^-(v)\bigr|+
\bigl|\{c_{i+2},c_{i+3}\}\cap N_D^+(v)\bigr|\leq3.
\]
Summing over all $i\in[m]$, every in-neighbor and every out-neighbor of $v$ on $C$ is counted exactly twice. Thus $2d_D(v,V(C))\leq3m$, a contradiction. Therefore all four required arcs are present for some $i$, and inserting $v$ at that position produces a $2$-cycle of order $m+1$.

Suppose that $r>0$, and choose $v\in V(D)\setminus V(C)$. There are at most $2(r-1)$ arcs between $v$ and the other uncovered vertices. Hence
\[
d_D(v,V(C))\geq\delta(D)-2(r-1)\geq\frac{8n}{5}-2r,
\]
where the last inequality follows from $\tau(n)\geq8n/5-2$. Since $m=n-r$ and $r<n/5$, we have that
\[
\frac{8n}{5}-2r-\frac{3m}{2}=\frac{n}{10}-\frac r2>0.
\]
Thus $d_D(v,V(C))>3m/2$, so $v$ can be inserted into $C$, contradicting the maximality of $C$. Hence $r=0$, and $C$ is the square of a Hamilton cycle.
\end{proof}

\begin{proof}[\textbf{Proof of Theorem~\ref{szc2}}]
Choose constants satisfying $0<1/n_0\ll\eta_0\ll\gamma\ll\xi\ll1/5$, with $n_0$ large enough for Lemma~\ref{nonextremal-cycle} and Theorem~\ref{acycle}. Let $D$ be an $n$-vertex digraph with $n\geq n_0$ and $\delta(D)\geq\tau(n)$.

If $D$ is not $(1/5,\gamma)$-extremal, Lemma~\ref{nonextremal-cycle} gives a $2$-cycle covering all but at most $\xi n$ vertices. If $D$ is $(1/5,\gamma)$-extremal, the same conclusion follows from Theorem~\ref{acycle}. Since $\xi<1/5$, in either case the resulting $2$-cycle leaves fewer than $n/5$ vertices uncovered. Lemma~\ref{completion-lemma} therefore yields the square of a Hamilton cycle in $D$.
\end{proof}

\section{Remarks}\label{Section6}

As the conclusion of this paper, we turn to factor problems. Recall that the $k$th power of a digraph $D$ is obtained by adding the arc $xy$ whenever $D$ contains a directed path from $x$ to $y$ of length at most $k$. Let $C_l^k$ denote the $k$th power of a directed cycle on $l$ vertices. A \emph{$C_l^k$-tiling} is a collection of disjoint copies of $C_l^k$, and it is a \emph{$C_l^k$-factor} if it covers all vertices of the host digraph.

The general transfer theorem of Lo~\cite{Allan} already yields asymptotic factor results from corresponding spanning-subgraph theorems. In particular, for every fixed $\varepsilon>0$ and all sufficiently large $l$, every digraph $D$ satisfying $\delta(D)\ge(8/5+\varepsilon)|D|$ contains a $C_l^2$-factor whenever $l$ divides $|D|$. Thus the asymptotic factor problem is covered by Lo's result, whereas determining the exact factor threshold without an additional linear error term remains open. This motivates the following questions. \begin{question} For every $k\ge2$, does there exist $l_0=l_0(k)$ such that, for every $l\ge l_0$ and every $n=lm$, each $n$-vertex digraph $D$ satisfying $\delta(D)\ge\frac{2(k+2)}{k+3}n$ contains a $C_l^k$-factor?\end{question}

A natural variant asks for a corresponding minimum semi-degree condition.

\begin{question}
For every $k\ge2$, does there exist $l_0=l_0(k)$ such that, for every $l\ge l_0$ and every $n=lm$, each $n$-vertex digraph $D$ satisfying $\delta^0(D)\ge\frac{k}{k+1}n$ contains a $C_l^k$-factor?
\end{question}

\bigskip

\noindent \textbf{Acknowledgments.} We are deeply grateful to Yangyang Cheng for his invaluable guidance and constructive suggestions concerning the analysis of the extremal case.

\end{document}